\documentclass[english]{amsart}
\usepackage[english]{babel}
\usepackage{csquotes}
\usepackage{bbm}
\usepackage{hyperref}
\usepackage{verbatim}

\usepackage{amsfonts,amssymb,mathtools}
\usepackage{amsmath}
\usepackage{mathtools}

\usepackage{graphicx,xcolor}
\graphicspath{{Figures/}}
\usepackage{tikz}

\usepackage[style=numeric,maxcitenames=2,maxbibnames=10,doi=false,isbn=false,url=false,natbib=true,giveninits=true,hyperref,bibencoding=utf8,backend=biber]{biblatex}
\AtEveryBibitem{%
  \clearfield{note}%
  \clearfield{issn}%
  \clearlist{language}%
  }

\newcommand{\N}{\ensuremath{\mathbb{N}}}
\newcommand{\R}{\ensuremath{\mathbb{R}}}

\newcommand{\E}{\ensuremath{\mathbb{E}}}
\renewcommand{\P}{\ensuremath{\mathbb{P}}}

\newcommand{\ind}[1]{\ensuremath{\mathbbm{1}_{\left\{#1\right\}}}}
\newcommand{\diff}{\mathop{}\mathopen{}\mathrm{d}}
\newcommand{\cal}[1]{\ensuremath{\mathcal{#1}}}
\newcommand\croc[1]{\left\langle #1\right\rangle}
\newcommand\steq[1]{\stackrel{\text{\rm #1.}}{=}}

\def\eps{\varepsilon}
\def\cadlag{c\`adl\`ag }

\newtheorem{proposition}{Proposition}
\newtheorem{definition}[proposition]{Definition}
\newtheorem{lemma}[proposition]{Lemma}
\newtheorem{theorem}[proposition]{Theorem}
\newtheorem{corollary}[proposition]{Corollary}

\title{A Stochastic Analysis of Particle Systems with Pairing}

\date{\today}
\author[V. Fromion]{Vincent Fromion}
\email{Vincent.Fromion@inrae.fr}
\address[V.~Fromion]{INRAE, MaIAGE, Universit\'e Paris-Saclay, Domaine de Vilvert, 78350 Jouy-en-Josas, France}
\author[Ph. Robert]{Philippe Robert}
\email{Philippe.Robert@inria.fr}
\urladdr{http://www-rocq.inria.fr/who/Philippe.Robert}
\author[J. Zaherddine]{Jana Zaherddine}
\email{Jana.Zaherddine@inria.fr}
\address[Ph.~Robert, J. Zaherddine]{INRIA Paris, 2 rue Simone Iff, 75589 Paris Cedex 12, France}

\begin{document}

\begin{abstract}
Motivated by a general principle governing numerous regulation mechanisms in biological cells,  we investigate a general interaction scheme between  different populations of particles and specific particles, referred to as agents. Assuming that each particle follows a random path in the medium, when a particle and an agent  meet, they may bind and  form a pair which has some specific functional properties. Such a  pair is also subject to random events and it splits after some random amount of time. In a stochastic context, using a Markovian model for the vector of the number of paired particles, and by taking the total number of particles as a scaling parameter, we study the asymptotic behavior of the time evolution of the number of paired particles. Two scenarios are investigated:  one with a large but fixed number of agents, and the other one, the dynamic case, when agents are created at a bounded rate and  may die after some time when they are not paired.  A first order limit theorem is established for the time evolution of the system in both cases.  The proof of an averaging principle of the dynamic case is one of the main contributions of the paper.  Limit theorems for fluctuations are obtained in the case of a fixed number agents.  The impact of dynamical arrivals of agents  on the level of pairing of the system is discussed. 
\end{abstract}

\maketitle

 \vspace{-5mm}

\bigskip

\hrule

\vspace{-3mm}

\tableofcontents

\vspace{-1cm}

\hrule

\bigskip

\section{Introduction}
In this paper we investigate a general mechanism of interaction between  different populations of particles and specific particles, agents,  in some environment. Assuming that each of the particles follows a random path in the medium, when a particle and an agent  meet, they may form a pair which has a specific functional property in the medium. Such a  pair is also subject to random events, it splits after some random amount of time. The efficiency of the pairing mechanism is analyzed with the time evolution of the number of paired particles of each type. 

\subsection{Motivation}
The initial motivation comes from molecular biology where this is an almost ubiquitous phenomenon occurring  in biological cells.  It can be (roughly) described as follows:   different types of macro-molecules (ribosomes, or polymerases for example), referred to as {\em particles}, are in charge of producing some of the functional components necessary to the development of the cell (mRNAs, proteins). Specific macro-molecules, referred to as {\em agents} in the paper, like small RNAs, have a regulation role in the cell. Agents can pair/bind with  particles to block, or to speed-up, their activity. Due to thermal noise, a pair agent-particle splits after some time. The dynamic behavior of the systems investigated are described in terms of binding/unbinding operations of  agents and particles.  See Section~\ref{BioSec} of the Appendix for a more detailed presentation of these aspects. 

\subsection{Literature}
A typical representation of pairing mechanisms in the literature, written as a chemical reaction, is of the type, 
\begin{equation}\label{MM}
{\cal Z}{+}{\cal F}_j \xrightleftharpoons{} {\cal ZF}_j\rightharpoonup{\cal G}_j{+}{\cal Z}
\end{equation}
where the chemical species are as follows: ${\cal Z}$ is  associated to what we call agents (enzymes, small RNAs, \ldots), ${\cal F}_j$ is for particles of type $j{\in}\{1,\ldots,J\}$ (RNAs, polymerases, \ldots). The species ${\cal ZF}_j$ is for pairs of ${\cal Z}$ and ${\cal F}_j$  and ${\cal G}_j$ is for a ``product'' of type $j$, it can be  ${\cal F}_j$.
In a deterministic setting the leads to a set of ODEs for a dynamical system $(X_{A}(t),A{\in} Z, F_j, ZF_j, G_j)$, for  example, for $(X_{{\scriptscriptstyle ZF_j}}(t))$ it gives
\begin{equation}\label{DetP}
\frac{\diff}{\diff t} X_{{\scriptscriptstyle ZF_j}}(t)= \kappa^+_jX_{{\scriptscriptstyle Z}}(t)X_{{\scriptscriptstyle F_j}}(t)- \kappa^-_jX_{{\scriptscriptstyle XF_j}}(t)
\end{equation}
for some constants $\kappa^\pm_j{\ge}0$. Note the quadratic term on the right hand side. Investigations are generally on the stability of these dynamical systems. See~\citet{Petrides}, \citet{Giudice} and~\citet{Jaya}. See Section~\ref{CRNSec} for a brief presentation of this formalism. 

In a stochastic context, this is represented as a Markov process  whose state descriptor is the vector of the number of copies of the different chemical species.  Simulations and numerical analysis of the associated Fokker-Planck equations have been used to study these phenomena, see~\citet{Petrides}.

The technical context is related to the celebrated Michaelis-Menten kinetics. These chemical reactions involve enzyme, substrate and product macro-molecules, whose associated chemical species are denoted respectively as ${\cal E}$, ${\cal S}$ and ${\cal P}$. The chemical reaction 
\[
  {\cal E}{+}{\cal S} \xrightleftharpoons{} {\cal ES}\rightharpoonup{\cal P}{+}{\cal E},
\]
has been investigated for some time now. The basic assumption for these models is that there are few copies of chemical species ${\cal E}$ but a large number of copies of substrate, so that the reaction rate is large (for the chemical reaction on the left). In a deterministic setting it leads to a system of non-polynomial ODEs. In a stochastic context,  these ODEs can be obtained via the proof of an averaging principle. See~\citet{Michaelis} and for a general overview~\citet{Sanft2011} and~\citet{Cornish}.  

Averaging principles also play an important role in our paper.  For the mathematical point of view, agents  may be seen as playing the role of enzymes in our model. Nevertheless our framework is not really that of  Michaelis-Menten. Their number is nevertheless not fixed in our main model of Section~\ref{OpenSecF}. As we will see, the production of agents has a strong impact on the qualitative behavior of the system. As it can be expected, the quadratic expressions due to  pairing mechanisms, like in Relation~\eqref{DetP}, are at the origin of some technical difficulties in the proof of limit theorems.

\subsection{Stochastic Model}
There are $J$ types of particles. For $1{\le}j{\le}J$, $N_j$ is the total number of particles of type $j$, this quantity is assumed to be fixed.  The total number of particles is $N{=}N_1{+}{\cdots}N_J$, it is our scaling parameter. 
A Markovian stochastic model is considered, each event occurs after an amount of time with an exponential distribution and the corresponding random variables are assumed to be independent.

There is only type of agent. An agent and a particle of type $j{\in}\{1,\ldots,J\}$ bind/pair at rate $\lambda_j$, and, in a reverse operation,  such a pair split into an agent and a particle of type $j$ at rate $\eta_j$. An agent or a particle which is not paired is said to be {\em free}.

The variables of interest are
\[
(F_N(t),Z_N(t))\steq{def} \left(\left(F_{N,j}(t), j{=}1,\ldots,J\right),Z_N(t)\right)
\]
where, for $1{\le}j{\le}J$ and $t{\ge}0$, $F_{j}(t)$ is the number of free particles of type $j$, i.e. not paired with an agent, and $(Z_N(t))$ is the process for the number of free agents. When the goal of pairing mechanism is of reducing the activity of the particles, this will be referred to as {\em sequestration} of particles, the objective is of minimizing
\[
\left(\sum_{j=1}^J \frac{F_{N,j}(t)}{N}\right)
\]
the process of the fraction of the  number of free particles. 
We analyze the asymptotic behavior, when $N$ goes to infinity, of the time evolution of  the $J$-dimensional process $(F_{j,N}(t)/N)$ associated to the free particles. An appropriate timescale for a non-trivial asymptotic evolution when $N$ goes to infinity has to be determined.

In~\citet{Fromion2}, a related model of sequestration has been analyzed, to study the regulation of transcription. It also includes  additional variables which are not considered in this paper. A component of the  stochastic model is  a related Markov process but in dimension $1$, i.e. for $J{=}1$. As it will be seen, compared to the case $J{=}1$, the multi-dimensional aspect of our model has a significant impact on the scaling properties of the associated stochastic processes.

Two types of models for agents  are analyzed.
\begin{enumerate}
\item  Agents are neither created nor removed: the number of agents is fixed,  of the order of $N$.
\item An agent is created at rate $\beta$ and, only when it is not paired with a particle, it dies at rate $\delta$.
\end{enumerate}
Case a) is used to investigate the case when the environment does not change significantly and when there is already a large number of agents to regulate the system.
Two cases are considered. In Section~\ref{SCritSeq} the total number of agents is of the order of $rN$ with $r{\in}(0,1)$, there are much more particles than agents. It is shown that
the process $(F_{j,N}(t)/N)$ is converging in distribution to the solution of an ODE. The equilibrium point of this ODE is unique and its coordinates are positive. For this system the number of free particles of type $j{\in}\{1,\ldots,J\}$ is, of course,  of the order of $N$.

In Section~\ref{CritSeq} the total number of agents is $N$ the same as the total number of particles. It is shown that with, appropriate initial conditions, the process $(F_{j,N}(t/\sqrt{N})/\sqrt{N})$ is converging in distribution to the solution of an ODE and a central limit theorem is proved, it shows that the fluctuations are of the order of $\sqrt[4]{N}$. In this case the impact of stochasticity on the pairing mechanism is minimal since there is a fraction of the order of $1/\sqrt{N}$ of free particles. 

Case b) is investigated in Section~\ref{OpenSecF} for the case when, initially, there  few agents (free or paired) are in the system,  the goal is of investigate the growth of the number of paired particles. The proof of an averaging principle in this context is challenging for several reasons. 

Since a paired agent does not die (it is not degraded), one can  expect an asymptotic situation as in Section~\ref{CritSeq} with a negligible fraction of free particles. We show that this is not the case, in fact, formally, the behavior is similar to that of Section~\ref{SCritSeq}, but on a faster time scale and with important qualitative and technical differences.

If the system starts with few agents, in this case most of $N$  particles are initially ``free'',  all agents created  will pair with a free particle right away and will keep doing that, via the successive steps of pairing/splitting,  as long the number of free particles is ``large'' so that, with high probability,  pairing occurs before degradation for agents. 
Given the rate of creation of agents, the natural timescale to study this problem is $(Nt)$. 

It can be expected that the  multi-dimensional process
\[
\left((\frac{F_N(Nt)}{N}\right){=}\left(\frac{F_{j,N}(Nt)}{N}\right)
\]
converges in distribution to a continuous process reflecting the asymptotic degree of pairing of the system. Due to their large transition rates, the integer-valued processes $(Z_N(Nt))$ and $(F_N(Nt))$ are  ``fast'' processes. Because of the space scaling,  $\left({F_N(Nt)}/{N}\right)$ is an a priori ``slow'' process.  Following the classical approach in this domain, see~\citet{Papanicolaou} in a stochastic calculus context and~\citet{Kurtz} for its formulation for jump process. For $T{>}0$, one has to consider the occupation measure associated to $(Z_N(Nt))$, i.e. this is the functional on non-negative Borelian functions on $[0,T]{\times}\N$,
\[
g\longrightarrow \int_0^T g(s,Z_N(Ns))\diff s.
\]
If this approach allows us to derive the results of Section~\ref{SCritSeq} for case a), where an averaging principle is proved, it does not work for case b). The sequence of 
processes $({F_{j,N}(Nt)}/{N},j{=}1,\ldots,J)$ does {\em not} converge in distribution in fact. It is not tight for the topology associated to uniform convergence if the initial state does not converge to some one-dimensional curve of $[0,1]^J$.  The main convergence result of this case is Theorem~\ref{TheoLLN}  of Section~\ref{OpenSecF}. 
It shows that the process associated to the total number of free particles,
\[
(\|F_N(Nt)\|)\steq{def} \left(\sum_{j=1}^J \frac{F_{N,j}(Nt)}{N}\right),
\]
converges in distribution to a continuous process. The sequence of $[0,1]^J$-valued processes $({F_{j,N}(Nt)}/{N},1{\le}j{\le}J)$ converges in distribution, but in a weak form, via its associated occupation measure.  It turns out that the process $(\|F_N(Nt)\|)$ determines, in some way, the behavior of the coordinates of $(F_N(Nt))$.  For $N$ large $(F_{N,j}(Nt))$ can in fact be represented as a curve of $[0,1]^J$ determined by $\|F_N(Nt)\|$.  In~\citet{Fromion2}, no such difficulty shows up since $J{=}1$.

Intuitively, it is shown that, in the limit, the number of free particles is of the order of $N$, as in case a) but for some specific $r{<}1$.  These results stress the impact of dynamical arrivals and  departures of agents. In particular the fraction  of paired particles  is asymptotically strictly less than $1$.  

Technical difficulties are related to the lack of tightness properties of the process $({F_{j,N}(Nt)}/{N})$. For this reason  the definition of the  occupation measure is extended to include also  ``slow''  processes and not only the fast processes as it is classical in the context of averaging principles.  As a functional on Borelian functions on $[0,T]{\times}[0,1]^J{\times}\N$, the occupation measure is expressed as  
\[
g\longrightarrow \int_0^T g\left(s,\frac{F_N(Ns)}{N},Z_N(Ns)\right)\diff s.
\]
The investigation of the limiting behavior of this sequence of occupation measures is the main topic of Section~\ref{OpenSecF}, including the identification of possible limiting points.

The reason of this behavior is essentially due to the interaction of several fast time scales. At the normal time scale $(t)$, if the components of the vector $F_N(t)$ are already of the order of $N$, the pairing/splitting events occur at a rate proportional to $N$. Since the natural time scaling for case b) is sped-up as $(Nt)$, roughly speaking, the pairing/splitting events will be instantaneously at equilibrium, at the first order, at any ``time'' $t$ for the current ``mass'' $\|F_N(Nt)\|$. In particular, if the initial point $\overline{F}_N(0)$ does not converge  to the equilibrium associated to the mass  $\|F_N(0)\|$, there cannot be a convergence in a neighborhood of $t{=}0$, this is the one-dimensional curve mentioned above.

\subsection*{Outline of the Paper}
Section~\ref{ModelSec} introduces notations and the Markovian process used to investigate pairing mechanisms. 
Section~\ref{CloseSec} analyzes the static case when the number of agents is fixed and in Section~\ref{OpenSecF} a stochastic averaging principle is proved when agents are created and degraded. To motivate  the design of such stochastic models,  Section~\ref{BioSec} of the appendix presents several examples  of regulation mechanisms in biological cells. Section~\ref{TechApp} of the appendix  is a quick reminder of classical limit results for $M/M/1$ and $M/M/\infty$ queues. These queues play an important role in the design of couplings used in the proofs of our limit theorems.

\section{Stochastic Model}\label{ModelSec}

\subsection*{Definitions and Notations}
If $H$ is a locally compact metric space, ${\cal C}_c(H)$ is the space of continuous functions with compact support endowed with the topology of uniform convergence. We denote by  ${\cal M}^+(H)$  the set of non-negative Radon measures on $H$ and ${\cal M}_1(H)$, the set of probability distributions on $H$, both spaces are endowed with the weak topology. See~\citet{Rudin}. Throughout the paper convergence in distribution of a sequence of jump processes $(U_N(t))$ to a process $(U(t))$ is understood with respect to the topology of  uniform convergence on compact sets for \cadlag functions. See Chapter~2 of~\citet{Billingsley} for example. 

For $J{\in}\N$, if $x{=}(x_i)$, $y{=}(y_i){\in}\R^J$, define
\begin{equation}\label{DefS}
\|x\| \steq{def} |x_1|{+}\cdots{+}|x_J|, \quad \croc{x,y}=x_1y_1{+}\cdots{+}x_Jy_J,
\end{equation}
and,
\begin{equation}\label{overeta}
\overline{x}{=}\max(x_j,1{\le}j{\le}J)\text{ and  }\underline{x}{=}\inf(x_j,1{\le}j{\le}J).
\end{equation}

We now introduce the main definitions for our stochastic model.  There are $J$ different types of particles.  The total number of particles of type $j{\in}\{1,\ldots,J\}$  is $C_{j,N}$, and $N{=}C_{1,N}{+}\cdots{+}C_{J,N}$, the total number of particles is a fixed number, it is also our scaling parameter.  It is assumed that, 
\begin{equation}\label{ScaleN}
\lim_{N\to+\infty} \left(\frac{C_{j,N}}{N}\right)=c{=}(c_j )
\end{equation}
holds, for some $c{\in}(0,1)^J$ such that $c_1{+}c_2{+}\cdots{+}c_J{=}1$.

\subsection*{State Space}
The state space is
\[
  {\cal S}_N{=}\left\{x=(f,z)=((f_j),z){\in}\N^{J+1}: f_j{\le}C_{j,N}, \forall 1{\le}j{\le}J\right\},
  \]
and, if $t{\ge}0$,
\begin{itemize}
\item for $1{\le}j{\le}J$, $F_{j,N}(t)$ denotes the number of free  particles of type $j$ at time $t$ and $F_N(t){=}(F_{j,N}(t),1{\le}j{\le}J)$;
\item The number of free agents  at time $t$ is  $Z_N(t)$;
\item The state of the process at time $t$ is $X_N(t){=}(F_N(t),Z_N(t)){\in}{\cal S}_N$.
\end{itemize}
The number of  agents paired with a particle of type $j$ at time $t$ is therefore $S_{j,N}(t){=}C_{j,N}{-}F_{j,N}(t)$.
In state $(F_N(t),Z_N(t))$, the total number of free particles is $\|F_N(t)\|$.

\subsection*{Transitions}
The dynamical behavior of $(X_N(t))$ is driven by several types of transitions.
\begin{enumerate}
\item A given particle of type $j$ and a given agent are paired at rate $\lambda_j$;
\item A pair (particle of type $j$, agent) is split at rate $\eta_j{>}0$ to give a particle of type $j$ and a free agent;
\item Agents are created at rate $\beta{\ge}0$ and a {\em free} agent dies, is degraded,  at rate $\delta{>}0$. An agent paired to a particle cannot die.
\end{enumerate}

The state process $(X_N(t)){=}(F_N(t),Z_N(t))$  is almost surely a {\em \cadlag function}, i.e.  a right-continuous function with left limits at any point of $(0,{+}\infty)$. It is described as an irreducible Markov process on ${\cal S}_N$ whose $Q$-matrix $Q_F$  is given by 
\[
(f,z){=}((f_j),z)\longrightarrow (f,z){+}
\begin{cases}
  ({-}e_j,{-}1) & \lambda_j f_j z, \\
  ({+}e_j,{+}1) & \eta_j(C_{j,N}{-}f_j), \\
  (0,{+}1) & \beta, \\
  (0,{-}1) & \delta z, \\
\end{cases}
\]
where $e_j$ is the $j$th unit vector of $\N^{J}$.

Note that the pairing mechanism induces quadratic transition rates in the $Q$-matrix.

\begin{definition}\label{rho}
\[
  c{=}(c_j),\, \eta{=}(\eta_j),\, \lambda{=}(\lambda_j), \quad 
  \rho_0{=}\frac{\beta}{\delta} \text{ and }
  \rho_j{=}\frac{\eta_j}{\lambda_j}, j{=}1,\ldots,J.
  \]
For $y{\in}(0,1)$,  $\phi(y)$ is defined as the unique solution of the equation
\begin{equation}\label{phi}
\sum_{j=1}^J \frac{\rho_j}{\rho_j{+}\phi(y)}c_j=y.
\end{equation}
\end{definition}

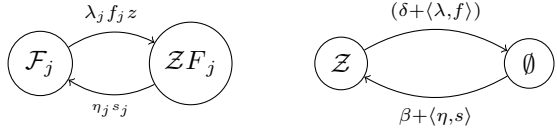
\begin{figure}[ht]
\begin{tikzpicture}[->,node distance=8mm] \node[circle,draw](F) at (0,0){${\cal F}_j$};
  \node[circle,draw](S) at (2,0){${\cal ZF}_j$};
  \node[circle,draw](Z) at (4,0){${\cal Z}$};
  \node[circle,draw](O) at (6.5,0){$\emptyset$};
  
  \path (F) edge [bend left,midway,above] node {$\scriptstyle{\lambda_j f_jz }$} (S);
  \path (S) edge [bend left,midway,below] node {$\scriptscriptstyle{\eta_j s_j}$} (F);
  \path (Z) edge [bend left,midway,above] node {$\scriptstyle{(\delta {+}\croc{\lambda,f})}$} (O);
  \path (O) edge [bend left,midway,below] node {$\scriptstyle{\beta{+}\croc{\eta,s}}$}(Z);
\end{tikzpicture}
\caption{Transitions of Pairing Mechanism, with $s{=}(s_j){\steq{def}}(C_{j,N}{-}f_{j})$.}
\end{figure}

\subsection*{Stochastic Differential Equations}
The process  $(X_N(t)){=}(F_N(t),Z_N(t))$ is represented as the solution of a system of SDEs (Stochastic Differential Equations). On the probability space there are $2(J{+}1)$ independent Poisson processes on $\R_+^2$ with intensity measure $\diff x{\otimes}\diff t$,  ${\cal P}_z^+$, ${\cal P}_z^-$, ${\cal P}_j^+$, ${\cal P}_j^-$,  $j{=}1,\ldots,J\}$. See~\citet{Rogers2}  for example.
The underlying filtration $({\cal F}_t)$ is defined by, for $t{>}0$,
\[
{\cal F}_t=\sigma\croc{{\cal P}_{j/z}^\pm([a,b]{\times}[0,s], j{=}1,\ldots,J\}, a\leq b, s{\le}t}.
\]
In the following,  measurability properties are assumed to be with respect to this filtration. 

Let $(F_N(t),Z_N(t))$ be the solution of the SDE, for $j{=}1,\ldots,J$, 
\begin{align}
\diff F_{j,N}(t) &= {\cal P}_{j}^+((0,\eta_j (C_{j,N}{-}F_{j,N}(t{-}))),\diff t){-}{\cal P}_{j}^-((0,\lambda_j F_{j,N}(t{-})Z_N(t{-})),\diff t),\label{SDEd1}\\
\diff Z_N(t) &= {\cal P}_{z}^+((0,\beta),\diff t){-}{\cal P}_z^{-}((0,\delta Z_N(t{-})),\diff t)\label{SDEd2}{+}\sum_{j=1}^J \diff F_{j,N}(t),
\end{align}
where $U(t{-})$ denotes the left-limit of the \cadlag process $(U(s))$ at $t{>}0$ and  with the usual notation,  if $A{\ge}0$ and  ${\cal P}$ is a Poisson point process on $\R_+^2$,
\begin{equation}\label{ConvInt}
 {\cal P}((0,A),\diff t)=\int\ind{x{\le}A} {\cal P}(\diff x,\diff t).
\end{equation}
By integrating these relations, we obtain, for $j{=}1,\ldots,J$, 
\begin{multline}\label{SDE1}
F_{j,N}(t)= F_{j,N}(0){+}M_{j,N}(t)\\{+}\eta_j\int_{0}^t  \left(C_{j,N}{-}F_{j,N}(s)\right)\diff s{-}\lambda_j \int_0^t F_{j,N}(s)Z_N(s)\diff s,
\end{multline}
and
\begin{multline}\label{SDE2}
  Z_N(t) = Z_N(0){+}M_{z,N}(t){+}\beta t{-}\delta \int_0^t  Z_N(s)\diff s\\
  {+}\sum_{j=1}^J \eta_j\int_{0}^t  (C_{j,N}{-}F_{j,N}(s)\diff s{-}\lambda_j\int_0^t F_{j,N}(s)Z_N(s)\diff s,
\end{multline}
where $(M_{j,N}(t))$ and $(M_{z,N}(t))$ are square integrable martingales whose previsible increasing processes are given by
\begin{align}
\left(\croc{M_{j,N}}(t)\right)&=\left(\eta_j\int_{0}^t  \left(C_{j,N}{-}F_{j,N}(s)\right)\diff s{+}\lambda_j \int_0^t F_{j,N}(s)Z_N(s)\diff s\right),\label{CrocSDE1}\\
\left(\croc{M_{z,N}}(t)\right)&=\left(\beta t{+}\delta \int_0^t  Z_N(s)\diff s {+}\sum_{j=1}^J \croc{M_{j,N}}(t)\right).\label{CrocSDE2}
\end{align}

\subsection*{Invariant Distribution}\label{CRNSec}
As explained in the introduction, our model can be expressed in the framework of chemical reaction networks (CRN).  Since we are mainly interested in the transient behavior of our system, we just give a quick sketch. It is only mentioned as an interesting aspect of our system. See~\citet{Feinberg} for a general introduction on CRNs.   

The corresponding chemical reactions are represented as
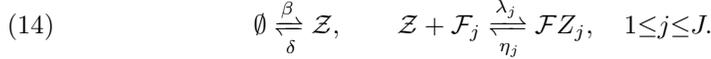
\begin{equation}\label{CRN}
\emptyset\mathrel{\mathop{\xrightleftharpoons[\delta]{\beta}}} {\cal Z},\qquad
{\cal Z}+{\cal F}_j\mathrel{\mathop{\xrightleftharpoons[\eta_j]{\lambda_j}}} {\cal FZ_j}, \quad 1{\le}j{\le}J.
\end{equation}
The associated dynamical system $((f_j(t),g_j(t)),z(t))$ is defined by the ODEs
\[
\begin{cases}
  \dot{f}_j(t)=\lambda_j f_j(t)z(t){-}\eta_j g_j(t),\quad 1{\le}j{\le}J,\vspace{2mm} \\ 
  \dot{f}_j(t){+}\dot{g}_j(t)=0, \quad 1{\le}j{\le}J,\vspace{2mm}\\
  \dot{z}(t)=\beta{-}\delta z(t).
\end{cases}
\]
Its fixed point is given by $((u_j,v_j),w)$ with
\begin{equation}\label{eqCRN}
w =\rho_0,\quad u_j=\frac{C^N_j}{1{+}\rho_0\rho_j}=C^N_j{-}v_j,\quad 1{\le}j{\le}J,
\end{equation}
with $\rho_0{=}\beta/\delta$ and $\rho_j{=}\lambda_j/\eta_j$, for $1{\le}j{\le}J$.

The characteristics of this CRN  are:
\begin{itemize}
\item $m{=}2J{+}2$ chemical species: ${\cal Z}$, ${\cal F}_j$, ${\cal FZ_j}$, $j{=}1,\ldots,J$;
\item $\ell{=}J{+}1$ cycles, these are the {\em single linkage classes} of the CRN;
\item The range, the dimension of the stochiometric space,  is $s{=}J{+}1$.
\end{itemize}
This is a CRN with deficiency $\delta{=}m{-}\ell{-}s{=}0$. A standard result, see~\citet{Anderson2010}, gives an explicit expression of the invariant distribution $\pi_N$ of $(X_N(t))$ on ${\cal S}_N$. 
\begin{proposition}\label{Def0}
  The invariant distribution of $(X_N(t))$ on ${\cal S}_N$ is given by, 
  \[
  \pi(f,s)=\frac{1}{Z_N}\frac{w^z}{z!}\prod_{j=1}^J\frac{u_j^{f_j}}{f_j!}\frac{v_j^{C^N_j-f_j}}{(C_j^N{-}f_j)!},\quad (f,z){=}((f_j),z){\in}{\cal S}_N,
  \]
  where $Z_N$ is the normalization constant and $(u_j)$ and $(v_j)$ are defined by Relation~\eqref{eqCRN}.
  \end{proposition}

\section{Fixed Number of Agents}\label{CloseSec}
Throughout this section the number of agents $C^N_Z$ is fixed, of the order of $r N$, with $r{<}1$ in Section~\ref{SCritSeq}, and  is exactly $N$, the total number of particles,  in Section~\ref{CritSeq}.  There are no creations or degradation of agents. Only pairing and splitting mechanisms operate in these cases.
As explained in the introduction, the purpose is of understanding the behavior of the system when the total number of agents does not change.  Section~\ref{OpenSecF} investigate a much more dynamic version of the system. 

The state space of the system is  ${\cal S}_N{\steq{def}}\prod_{j=1}^J\{0,\ldots,C_j^N\}$. For a state $x{=}(x_j){\in}{\cal S}_N$, the total number of paired particles with an agent is
$N{-}x_1{-}\cdots{-}x_J$. The associated process  in ${\cal S}_N$ is denoted as $(F^r_N(t))$ has the  Markov property,  its $Q$-matrix is given by, for $x{\in}{\cal S}_N$,
\[
x\rightarrow
\begin{cases}
x{+}e_j&\displaystyle\eta_j\left(C_j^N{-}x_j\right),\\
x{-}e_j &\displaystyle \lambda_j x_j\left(C^N_Z{-}(N{-}x_1{-}\cdots{-}x_J)\right),
\end{cases}
\]
where $e_j$ is the $j$th unit vector of $\N^J$.

\subsection{Overloaded Case}\label{SCritSeq}
In this section the total number of agents if of the order of $rN$, with $r{<}1$,
\begin{equation}\label{ScaCZ}
\lim_{N\to+\infty} \frac{C^N_Z}{N}=r. 
\end{equation}
Since there are not enough agents to handle all particles, it is  clear that the number of free particles of type $j$, $1{\le}j{\le}J$, should be of the order of $N$.  
The SDE~\eqref{SDE1} for  $(F^r_N(t))$ becomes, for $1{\le}j{\le}J$,
\begin{multline}\label{SDEe1}
\diff F^r_{N,j}(t) = {\cal P}_{S_j}((0,\eta_j (C_j^N{-}F_{N,j}(t{-}))),\diff t)\\{-}{\cal P}_{F_j}((0,\lambda_j F^r_{N,j}(t{-})Z_N^r(t{-})),\diff t),
\end{multline}
where $(Z^r_N(t))$ is the process of free agents,
\begin{equation}\label{cZ0eq}
\left(Z^r_N(t)\right)=\left(C^N_Z{-}\sum_{j=1}^J \left(C^N_j{-}F^r_{N,j}(t)\right)\right)
=\left(C^N_Z{+}\|F^r_{N}(t)\|{-}N\right). 
\end{equation}
We assume that the initial conditions are such that $Z_N(0){=}z_0$, for some fixed $z_0{\in}\N$, and
\begin{equation}\label{InitOp2}
\lim_{N\to+\infty} \frac{F^r_N(0)}{N}=\overline{f}_0=(\overline{f}_{0,j}){\in}[0,1]^J,
\end{equation}
such that
\begin{equation}\label{InitOp3}
\sum_{j=1}^J \overline{f}_{0,j}=1{-}r. 
\end{equation}
The last condition expresses simply that most of the agents are initially paired with particles. We will see that the number of free agents remains a finite random variable. 
\begin{definition}
For $N{>}0$, the scaled process is defined as 
\begin{equation}\label{Fbar0}
  \left(\overline{F}^r_N(t)\right)\steq{def}\left(\frac{F^r_{N,j}(t)}{N},j{=}1,\ldots,J\right).
\end{equation}
If $g$ is non-negative Borelian function on $\R_+{\times}\N$, we define the occupation measure
\begin{equation}\label{OccMeasDef0}
\croc{\Lambda^r_N,g}\steq{def} \int_{\R_+}g\left(s,Z^r_N(s)\right)\diff s.
\end{equation}
\end{definition}
Since $F^r_{N,j}(t){\le}C_j^N{\le}N$, $1{\le}j{\le}J$, for  $t{\ge}0$,
the state space of the process $(\overline{F}^r_N(t))$ is included in $[0,1]^J$. 

\begin{lemma}\label{LemCl1}
If $N$ is sufficiently large,  there exists a coupling of the process $(Z^r_N(t))$ with $(L(Nt))$, where $(L(t))$ is  an $M/M/\infty$ queue with input rate $\overline{\eta}$ and service rate $\underline{\lambda}r/{2}$, with
\[
\overline{\eta}=\max_j \eta_j, \text{ and }  \underline{\lambda}=\min_j \lambda_j,
\]
such that  $L(0){=}z_0$  and $Z^r_N(t){\le} L(Nt)$ holds for all $t{\ge}0$. 
\end{lemma}
See Section~\ref{MMI} of the appendix on the $M/M/\infty$ queue.
\begin{proof}
This is a simple consequence of the fact that if $Z^r_N(t){=}z{\in}\N$, the rate at which there is a jump of size ${+}1$, resp. ${-}1$, is
  \[
  \sum_{j=1}^N\eta_j\left(C^N_j{-}F^r_{N,j}(t)\right)\leq \overline{\eta}   \sum_{j=1}^NC^N_j=\overline{\eta}N,
  \]
  resp.
  \[
    \sum_{j=1}^N\lambda_jF^r_{N,j}(t) \geq \underline{\lambda} (N{-}C^N_Z)\ge \underline{\lambda}\left(1{-}\eps\right)N,
    \]
for some $\eps{\in}(0,1)$ if $N$ is large enough. It is then straightforward to construct the desired coupling. 
\end{proof}

The integration of the SDE~\eqref{SDEe1} gives the relation
\begin{multline}\label{eqf1}
\overline{F}^r_{j,N}(t) =\overline{F}^r_{j,N}(0){+}M^r_{j,N}(t)\\
- \lambda_j \int_0^{t}\overline{F}^r_{N,j}(s)Z^r_N(s)\diff s{+} \eta_j\int_0^{t}\left(\frac{C_j^N}{N}{-}\overline{F}^r_{N,j}(s)\right)\diff s,
\end{multline}

The process  $(M^r_{j,N}(t))$ is a martingale whose previsible increasing process is
\begin{equation}\label{croceqf1}
  \left(\croc{M^r_{j,N}}(t)\right)=\left(
\frac{\lambda_j}{N} \int_0^t \overline{F}^r_{j,N}(s)Z^r_N(s)\diff s{+}\frac{\eta_j}{N} \int_0^{t}\left(\frac{C_j^N}{N}{-}\overline{F}^r_{N,j}(s)\right)\diff s\right).
\end{equation}

\begin{proposition}\label{PropLLNCl}
Under the assumptions~\eqref{ScaleN},~\eqref{InitOp2}, and~\eqref{InitOp3}for the initial state,   then for the convergence in distribution the relation 
  \[
 \lim_{N\to+\infty} \left(\frac{F^r_N(t)}{N}\right)=(f^r(t))=(f^r_{j}(t)),
  \]
holds,   where $(x(t))$ is the solution of the ODEs, for $1{\le}j{\le}J$ and $t{>}0$,
\begin{equation}\label{EqLaClosed}
\frac{\diff}{\diff t} f^r_{j}(t)=
\lambda_j f^r_{j}(t)\frac{\croc{\eta,c{-}x(t)}}{\croc{\lambda,x(t)}} {-} \eta_j \left(c_j{-}f^r_{j}(t)\right),
\end{equation}
with Definition~\ref{rho}. 
\end{proposition}
\begin{proof}
By using the notations of Lemma~\ref{LemCl1} and the ergodic theorem for positive recurrent Markov processes, it is not difficult to prove that the sequence of processes
  \[
  \left(\int_0^tL(Ns)\diff s\right)
  \]
is tight with the criterion of the modulus of continuity, see Theorem~7.3 of~\citet{Billingsley}, and that its limiting point is  necessary $(\overline{\eta}/\underline{\lambda}\cdot t)$.

Since $F^r_{N,j}(t){\le}C_j^N{\le}N$, for $1{\le}j{\le}J$ and $t{\ge}0$, with Relation~\eqref{croceqf1}, we obtain therefore that the process $(\croc{M^r_{j,N}}(t))$ is converging in distribution to $0$, Doobs' Inequality gives that the same result holds for the martingale $(M^r_{j,N}(t))$.

For $T{>}0$, with Relation~\eqref{eqf1}, the modulus of continuity of $(\overline{F}^r_{j,N}(t))$ on the time interval $[0,T]$ is
\begin{multline*}
\omega_{F_N,T}(\delta)\steq{def}\sup_{\substack{s,t{\le}T\\|s-t|\le \delta}}\left|\overline{F}^r_{j,N}(t)){-}\overline{F}^r_{j,N}(s)\right|
\\ \le \sup_{s\le T}|M^r_{j,N}(s)|+\lambda_j\sup_{\substack{s\le t{\le}T\\|s-t|\le \delta}}\int_s^t L(Nu)\diff u+\delta\eta_j\frac{C_j^N}{N}. 
\end{multline*}
Again with Theorem~7.3 of~\citet{Billingsley}, we deduce that  the sequence of processes $(\overline{F}^r_{N,j}(t))$ is tight. 

For $K{>}0$,
\[
\E(\Lambda^r_N([0,T]{\times}[K,{+}\infty)))\leq \int_0^T\P\left(L(Ns){\ge}K\right)\diff s =\frac{1}{N}\int_0^{NT}\P(L(s){\ge}K)\diff s,
  \]
since $(L(s))$ converges in distribution to a Poisson distribution, see Section~\ref{MMI} of the appendix, the last term can be made arbitrarily small for $K$ sufficiently large. 
Lemma~1.3 of~\citet{Kurtz} gives the tightness of the sequence of random measures $(\Lambda^r_N)$  and any of its limiting points $\Lambda^r_{\infty}$ can be represented as,
\[
\croc{\Lambda^r_\infty,f}=\int_{\R_+{\times}[0,1]^J{\times}\N} f\left(s,x\right)\mu_s(\diff x)\diff s,
\]
if $f$ is a non-negative Borelian function on $[0,T]{\times}\N$, where $(\mu_s)$ is an optional process on ${\cal M}_1(\N)$.

Hence the sequence of random variables $((\overline{F}^r_N(t)), \Lambda^r_\infty)$ is tight, we denote by $((f^r(t)),\Lambda^r_{\infty})$ one of its limiting points. 
Since Proposition~\ref{pirep} establishes a similar result, but in a more difficult technical framework. The analogue of the sequence of processes $(\overline{F}^r_N(t))$ is {\em not} tight in this case. For this reason,  we skip the proof of the fact that, in the above representation of $\Lambda^r_\infty$, $\mu_s$ can be expressed as  a Poisson distribution with parameter
\[
\frac{\croc{\eta,c{-}f^r(s)}}{\croc{\lambda,f^r(s)}},
\]
and that $(f^r(t))$ satisfies the ODE~\eqref{EqLaClosed}. The proposition is proved. 
\end{proof}

\begin{corollary}
With the notations of Proposition~\ref{PropLLNCl}, the equilibrium point $f^r_{\infty}$ of $(f^r(t))$ is given by 
\[
f^r_{\infty}=\left(\frac{\eta_jc_j}{\eta_j{+}\lambda_j h_\infty}\right),
\]
where $h_\infty{=}\phi(1{-}r)$ and  $\phi$ is defined by Relation~\eqref{phi}. 
\end{corollary}

\subsection{Critical Case}\label{CritSeq}
The number of agents is exactly $N$, the total number of particles and there are still no creation or degradation of agents.  We prove that the process of the number of free particles   of type $j{\in}\{1,\ldots\}$, $(F^1_{N,j}(t))$, is of the order of $\sqrt{N}$, with fluctuations of the order of $\sqrt[4]{N}$. See Theorems~\ref{CLLN} and~\ref{CCLT}.

Note that, for $t{\ge}0$, the total number of free agents at time $t$ is 
\[
N-\sum_{j=1}^J \left(C_j^N{-}F^1_{N,j}(t)\right)=\sum_{j=1}^J F^1_{N,j}(t)=\|F^1_N(t)\|.
\]
The $Q$-matrix $Q_f$ of $(F^1_{N,j}(t))$  is thus given  by, for $x{\in}{\cal S}_N$,
\[
x\rightarrow
\begin{cases}
x{+}e_j&\displaystyle\eta_j\left(C_j^N{-}x_j\right),\\
x{-}e_j &\displaystyle \lambda_j x_j\|x\|. 
\end{cases}
\]
\begin{lemma}\label{Lem1Clo}
  If, for $\eta$, $\lambda{>}0$ and $N{\ge}1$, if $(X_N(t))$ is the solution of the SDE,
\[
\diff X_N(t)={\cal P}_{S_1}((0,\eta N),\diff t){-}{\cal P}_{F_1}\left(\left(0,\lambda X_N(t{-})^2\right),\diff t\right),
\]
with $X(0){=}0$, then for any $T{\ge}0$, there exists $K_1{>}0$ such that,
\[
\lim_{N\to+\infty}\P\left(\sup_{t{\le}NT}\frac{X_N(t)}{\sqrt{N}}\ge K_1 \right)=0,
\]
and
\[
\sup_{t{\ge}0}\E\left(X_N(t)^2\right)<{+}\infty. 
\]
\end{lemma}
\begin{proof}
We fix $K_0$ such that $\lambda K_0^2{>}\eta$. If we define the process  $(Y(t))$ by the SDE
\[
\diff Y(t)={\cal P}_{S_1}((0,\eta),\diff t){-}\ind{Y(t{-}){>}0}{\cal P}_{F_1}((0,\lambda K_0^2),\diff t),
\]
with $Y(0){=}0$. As  in the proof of Proposition~\ref{PropCoup}, by induction on the successive jumps of $(X_N(t))$,  it is easy to show that the relation
\[
X_N(t){\le}K_0\sqrt{N}{+}Y(Nt)
\]
holds almost surely for all $t{>}0$.
The process $(Y(t))$ is a reflected random walk on $\N$, it is usually associated to the $M/M/1$ queue. See Chapter~5 of ~\citet{Robert}. Proposition~5.11 of this reference gives that if $T_A$ is the hitting time of $A$ by $(Y(t))$ then the random variable
\[
\left(\frac{\eta}{\lambda K_0^2}\right)^A T_A
\]
is converging in distribution to an exponential random variable when $A$ goes to infinity. This shows in particular that $(T_A/A^2)$ is converging in distribution to infinity, hence, for any $K{>}0$,
\[
\lim_{N\to+\infty} \P\left(\sup_{t{\le}T}\frac{Y(tN)}{\sqrt{N}}\ge 1 \right)=
\lim_{N\to+\infty} \P\left(T_{\lceil \sqrt{N}\rceil}{\le}{TN} \right)=0. 
\]
This gives the first part of the lemma. We conclude the proof by setting $K_1{=}K_0{+}1$ and remarking that  the invariant distribution of $(Y(t))$ is a geometric distribution with parameter ${\eta}/{(\lambda K_0^2)}$ and that, with a simple coupling argument,  the mapping $t{\to}\E(Y(t)^2)$ is non-decreasing. 
\end{proof}
The next result shows that all coordinates of $(F^1_N(t))$ are at most of the order of $\sqrt{N}$ very quickly independently of the initial point. Theorem~\ref{CLLN} completes this result by showing that the order of magnitude of its  coordinates is exactly $\sqrt{N}$.
\begin{proposition}[Coupling]\label{PropCoup}
  For all $j{\in}\{1,\ldots,J\}$, 
\begin{equation}\label{CoupX}
 F^1_{N,j}(t)\leq    F^1_{N,j}(0){+}X_{N,j}(t), \quad t{\ge}0,
\end{equation}
where the $(X_{N,j}(t))$ are the solutions of the SDEs
\begin{equation}\label{eqXi}
\diff X_{N,j}(t)={\cal P}_{S_j}((0,\overline{\eta}N),\diff t){-}{\cal P}_{F_j}\left(\left(0,\underline{\lambda}X_{N,j}(t{-})^2\right),\diff t\right),\quad 1{\le}j{\le}J,
\end{equation}
 with $(X_j(0)){=}0$ and $\overline{\eta}$ and $\underline{\lambda}$ are defined by Relation~\eqref{overeta}. 

There exists  $\ell_0{>}0$ such that if
\[
\tau_N\steq{def} \inf\left\{t{>}0: F^1_{N,j}(t){\le}\left\lceil \ell_0\sqrt{N}\right\rceil,\forall j{\in}\{1,\ldots, J\}\right\},
\]
then
    \[
\sup_{N{\ge}1}  \sup_{x{\in}{\cal S}_N} \E_x(\tau_N)< {+}\infty. 
    \]
\end{proposition}
\begin{proof}
The $Q$-matrix $Q_X$ of the Markov  process $(X_{N,j}(t))$ defined by the SDEs~\eqref{eqXi} is
\[
x\rightarrow
\begin{cases}
x{+}e_j&\overline{\eta} N,\\
x{-}e_j & \underline{\lambda}x_j^2,\\
\end{cases}
\]
clearly $q_f(x,x{+}e_j){\le}q_X(x,x{+}e_j)$ and $q_f(x,x{-}e_j){\ge}q_X(x,x{-}e_j)$.

A simple coupling, by induction on the successive jumps of $(F^1_{N,j}(t))$,  gives that the relation
\[
 F^1_{N,j}(t)\leq F^1_{N,j}(0) {+}X_{N,j}(t)
\]
 holds for all $t{\ge}0$.

To prove the last assertion, in view of Relation~\eqref{CoupX}, it is enough to prove it for the ``maximal'' initial state, i.e.  $(F^1_{N,j}(0)){=}(C_j^N)$. 
If, for $A{>}0$, $\|x\|{>} J A \sqrt{N}$, 
 then, if $g(x){=}\|x\|$, for $x{\in}\N^J$, 
 \[
Q_f(g)(x)\leq J\overline{\eta}N{-}\underline{\lambda}J^2A^2N.
\]
If we choose $\ell{=}J A$  such that $\gamma{=}\underline{\lambda}\ell^2{-}J\overline{\eta}{>}0$, by using Proposition~8.14 and Theorem~8.13 of~\citet{Robert},  we obtain
\[
\E(\tau_N)\le \frac{g(F^1_N(0))}{N\gamma}{=}\frac{1}{\gamma}.
\]
The proposition is proved. 
\end{proof}

\begin{theorem}[Law of Large Numbers]\label{CLLN}
  If
  \[
  \lim_{N\to+\infty} \frac{F^1_N(0)}{\sqrt{N}} =\overline{f}^1_0{=}\left(\overline{f}^1_{0,j}\right), \text{ and }
  \left(\overline{F}_N(t)\right)\steq{def}\left(\frac{F^1_{N,j}\left(t/\sqrt{N}\right)}{\sqrt{N}}\right),
  \]
then the sequence of processes $\left(\overline{F}_N(t)\right)$ is converging in distribution to  
the solution $(\overline{f}^1(t)){=}(\overline{f}^1_j(t))$ of the ODE,
\begin{equation}\label{ell}
\left(\overline{f}^1_j\right)'(t)=c_j\eta_j  {-}\lambda_j\overline{f}^1_j(t)\left\|\overline{f}^1(t)\right\|,
\end{equation}
with $\overline{f}^1(0){=}\overline{f}^1_0$.

The equilibrium point of the ODE~\eqref{ell} is given by
\begin{equation}\label{EquClo}
\overline{f}^1_{j,\infty}=\left({c_j\rho_j}\left/\sqrt{c_1\rho_1{+}\cdots{+}c_J\rho_J}\right.\right),
\end{equation}
where $(\rho_j)$ is given by Definition~\ref{rho}. 
\end{theorem}
\begin{proof}
By integration of Relation~\eqref{SDECloo}, we obtain, for $t{\ge}0$, 
\begin{multline}\label{SDE2Clo}
\overline{F}_{N,j}(t)= \overline{F}_{N,j}(0){+}M_{N,j}^0(t)\\{+}\eta_j\int_0^t \left(\frac{C_j^N}{N}{-}\frac{\overline{F}_{N,j}(s)}{\sqrt{N}}\right)\diff s{-}\lambda_j\int_{0}^t \overline{F}_{N,j}(s)\sum_{k=1}^J \overline{F}_{N,k}(s)\diff s,
\end{multline}
where $(M_{N}^0(t)){=}(M_{N,j}^0(t),1{\le}j{\le}J)$ is the martingale defined by, for $1{\le}j{\le}J$, 
\begin{multline}
 M_{N,j}^0(t)\steq{def}\\\frac{1}{\sqrt{N}}\int_0^{t/\sqrt{N}} \left[{\cal P}_{S_j}((0,\eta_j(C_j^N{-}F^1_{N,j}(s{-}))),\diff s){-}\eta_j(C_j^N{-}F^1_{N,j}(s))\diff s\right]\\{-} 
\frac{1}{\sqrt{N}}\int_0^{t/\sqrt{N}} \hspace{-1mm}\left[{\cal P}_{F_j}\left(\left(0,\lambda_j F^1_{N,j}(s{-})\sum_{k=1}^J \overline{F}_{N,k}(s{-})\right),\diff s\hspace{-1mm}\right)\right.\\\rule{0mm}{5mm}\left.{-}\lambda_j F^1_{N,j}(s)\sum_{k=1}^J \overline{F}_{N,k}(s)\diff s\right].
\end{multline}
Its previsible increasing process is given by
\begin{multline}\label{crocMClo}
\left(\croc{M^0_{N,j}}(t)\right)=
\left(\frac{\eta_j}{\sqrt{N}}\int_0^t \left(\frac{C_j^N}{N}{-}\frac{\overline{F}_{N,j}(s)}{\sqrt{N}}\right)\diff s\right.\\\left.{+}\frac{\lambda_j}{\sqrt{N}}\int_{0}^t\overline{F}_{N,j}(s)\sum_{k=1}^J \overline{F}_{N,k}(s)\diff s\right),
\end{multline}
and $\croc{M^0_{N,j},M^0_{N,k}}(t){=}0$, for $1{\leq}j{\not=}k{\le}J$. 
Lemma~\ref{Lem1Clo} shows the convergence
\[
\lim_{N\to+\infty}\left(\E\left(\croc{M^0_{N,j}(t),M^0_{N,k}}(t)\right),1{\le}j,k{\le}J\right)=0,
\]
and, with Doob's Inequality, the martingale  $(M_N^0(t))$ converges to $0$, and also that, for the convergence in distribution
\[
\lim_{N\to+\infty}\left(\frac{\overline{F}_{N,j}(t)}{\sqrt{N}}\right)=0.
\]
Standard arguments, using the criterion of the modulus of continuity, see Theorem~7.3~\citet{Billingsley} for example,  give that the sequence of processes $(\overline{F}_N(t))$ is tight and  that any limiting point $(\overline{f}^1(t)){=}(\overline{f}^1_j(t))$ satisfies the identity
\begin{equation}\label{ODEell}
\overline{f}^1_j(t)=\overline{f}^1_{j}(0){+}\eta_j c_j t{-}\lambda_j\int_0^t\overline{f}^1_j(s)\sum_{k=1}^J \overline{f}^1_{k}(s)\diff s.
\end{equation}
The theorem is proved.
\end{proof}

The fluctuations of $(F^1_N(t))$ on the timescale $(t/\sqrt{N})$ are now developed in the following theorem. 
\begin{theorem}[Central Limit Theorem]\label{CCLT}
Under the assumption on the initial state of Theorem~\ref{CLLN}, if 
  \[
 \left(\widehat{F}^1_N(t)\right)= \left(\widehat{F}^1_{N,j}(t)\right)\steq{def}\left(\frac{F^1_{N,j}\left(t/\sqrt{N}\right){-}\sqrt{N}\,\overline{f}^1_j(t)}{\sqrt[4]{N}}\right),
  \]
where $\left(\overline{f}^1(t)\right)$ is defined by Relation~\eqref{ell} and
  \[
\lim_{N\to+\infty} \widehat{F}^1_N(0) = \widehat{f}^1_0\in\R^J,
\]
then the sequence of processes $(\widehat{F}^1_N(t))$ is converging in distribution to $(\widehat{F}^1(t))$, the solution of the SDE
\begin{multline}\label{SDECloo}
\diff \widehat{F}^1_j(t)= \sqrt{{-}\left(\overline{f}^1_{j}\right)'(t){+}2\eta_jc_j }\,\diff B_j(t)
\\{-}\lambda_j \left(\widehat{F}^1_{j}(t)\left\|\overline{f}^1(t)\right\|{+}\overline{f}^1_{j}(t)\left\|\widehat{F}^1(t)\right\|\right)\diff t,
\end{multline}
with $\widehat{F}^1(0){=}\widehat{f}^1_0$, where $(B_j(t))$ is the standard Brownian motion in $\R^J$. 
\end{theorem}
\begin{proof}
Relations~\eqref{SDECloo} and~\eqref{ODEell} give the identity,
\begin{multline}\label{eqb1}
\widehat{F}^1_{N,j}(t)=\sqrt[4]{N}\left(\overline{F}^1_{N,j}(t){-}\overline{f}^1_j(t)\right)= \widehat{F}^1_{N,j}(0){+}\sqrt[4]{N}M_{N,j}^0(t)\\{+}\eta_j\frac{C_j^N{-}c_jN}{N^{3/4}}t  {-}\eta_j\int_0^t \frac{\overline{F}^1_{N,j}(s)}{\sqrt[4]{N}}\diff s\\{-}\lambda_j\int_{0}^t \widehat{F}^1_{N,j}(s)\sum_{k=1}^J \overline{F}^1_{N,k}(s)\diff s{-}\lambda_j\int_{0}^t \overline{F}^1_{N,j}(s)\sum_{k=1}^J \widehat{F}^1_{N,k}(s)\diff s, 
\end{multline}
and, with Relation~\eqref{crocMClo}, for $1{\le}j{\le}J$, 
\begin{multline*}
\left(\croc{\sqrt[4]{N}M^0_{N,j}}(t)\right)=
\left(\eta_j\int_0^t \left(\frac{C_j^N}{N}{-}\frac{\overline{F}^1_{N,j}(s)}{\sqrt{N}}\right)\diff s\right.\\\left.{+}\lambda_j\int_{0}^t\overline{F}^1_{N,j}(s)\sum_{k=1}^J \overline{F}^1_{N,k}(s)\diff s\right).
\end{multline*}
From Lemma~\ref{Lem1Clo} and Theorem~\ref{CLLN}, we obtain that, for the convergence in distribution,  the relation
\begin{multline*}
\lim_{N\to+\infty}\left(\croc{\sqrt[4]{N}M^0_{N,j}}(t)\right)=
\left(\eta_jc_j t{+}\lambda_j\int_{0}^t\overline{f}^1_{j}(s)\sum_{k=1}^J \overline{f}^1_{k}(s)\diff s\right)\\
=\left(\overline{f}^1_{j}(0){-}\overline{f}^1_{j}(t){+}2\eta_jc_j t\right),
\end{multline*}
holds, by Relation~\eqref{ODEell}. Recall that
\[
\left(\croc{\sqrt[4]{N}M^0_{N,j},\sqrt[4]{N}M^0_{N,k}}(t)\right){=}0
\]
holds  for $1{\le}j{\ne}k{\le}J$.  Theorem 1.4 page~339 of~\citet{Ethier} shows that the sequence of martingales $(\widehat{M}_N^0(t))$ converges in distribution to the distribution of the process
\[
\left(\int_0^t \sqrt{{-}(\overline{f}^1_{j})'(s){+}2\eta_jc_j }B_j(\diff s)\right),
\]
where $(B_j(t))$ is a standard Brownian motion on $\R^J$. 
Using again Lemma~\eqref{Lem1Clo}, we have 
\[
\lim_{N\to+\infty}   \left(\int_0^t \frac{\overline{F}^1_{N,j}(s)}{\sqrt[4]{N}}\diff s\right)=(0). 
\]
The rest of the proof is standard, first by showing the tightness of $(\widehat{F}^1_N(t))$ and secondly by identifying it as the solution of an SDE. See the proof of Theorem~6.14 of~\cite{Robert} for example. The theorem is proved. 
\end{proof}
The following proposition shows that the  invariant distribution of the Markov process $(F^1_N(t))$ has in fact a simple expression. This is a consequence of Proposition~\ref{Def0}, the reversibility property is in fact the additional (simple) result. 
\begin{proposition}[Invariant Distribution]
The Markov process $(F^1_N(t))$ is re\-versi\-ble, and its invariant distribution $\pi_N$
\begin{equation}\label{InvClosed}
\pi_N(x)=\frac{1}{Z_N}\frac{1}{\|x\|!}\prod_{j=1}^J \rho_j^{x_j}\frac{C_j^N!}{(C_j^N{-}x_j)!x_j!},\qquad x{\in}{\cal S}_N,
\end{equation}
where $Z_N$ is a normalizing constant.
\end{proposition}
A version of Theorems~\eqref{CLLN} and~\eqref{CCLT} could probably be considered via a saddle-point analysis of the constant $Z_N$. This is not done in this paper. 

\section{Dynamical Arrivals}\label{OpenSecF}
If the systems starts with few agents so that most of $N$  particles are ``free'', when  an agent created, it  is paired with a free particle right away, at a rate proportional to $N$. This will happen repeatedly, via the successive steps of sequestration/de-sequestration,  as long the number of free particles is sufficiently ``large'' so that sequestration occurs always before  the degradation/death  of an agent.  The precise result is in fact a little more subtle than that.  We show that, in the limit, on the timescale $t{\mapsto}Nt$,  there remains a positive fraction of free particles of the order of $N$. 

The state descriptor of the pairing process is in this case
\[
(X_N(t)){=}(F_N(t),Z_N(t))=((F_{j,N}(t), j{=}1,\ldots, J),Z_N(t)).
\]
It can be expressed as the solution of  the SDEs~\eqref{SDEd1} and~\eqref{SDEd2}, the initial conditions are assumed to satisfy the following  scaling relations
\begin{equation}\label{InitOp}
\lim_{N\to+\infty} \frac{F_N(0)}{N}=\overline{f}_0\ne 0, \overline{f}_0{=}(\overline{f}_{0,j}){\in}\prod_{j=1}^J [0,c_j]^J, \text{ and } Z_N(0){=}z_0{\in}\N,
\end{equation}
where $c{=}(c_j)$ is defined by Relation~\eqref{ScaleN}.

Initially, a fraction  $\overline{f}_{0,j}$ of particles of type $j{\in}\{1,\ldots,J\}$ are  free and there are $z_0$ free agents. Since the external input rate of agents is constant and equal to $\beta$,  in order to have a positive fraction  in $N$ of particles paired with an agent, the natural time scale  to consider is, at least,  $t{\mapsto}Nt$.

The setting of the analysis will be that of averaging principles as presented in~\citet{Kurtz}. As it will be seen there are specific technical difficulties related to the scaling framework which we introduce now. 
\begin{definition}[Scaled Processes]\label{ScalDef}
 For $N{>}0$,   $(\overline{X}_N(t)){\steq{def}}(\overline{F}_N(t),Z_N(Nt))$, with 
\begin{equation}\label{Fbar}
\left(\overline{F}_N(t)\right)\steq{def}\left(\frac{F_N(N t)}{N}\right)=\left(\frac{F_{j,N}(N t)}{N},j{=}1,\ldots,J\right){\in}\prod_{j=1}^J\left[0,\frac{C_{j,N}}{N}\right].
\end{equation}
For $t{\ge}0$, we have $\|F_N(t)\|{\le}1$ since $F_{j,N}(t){\le}C_j^N$,for all $1{\le}j{\le}J$. 

The {\em occupation measure} is the random measure on $H{\steq{def}}\R_+{\times}[0,1]^{J}{\times}\N$ defined by
\begin{equation}\label{OccMeasDefN}
  \croc{\Lambda_N,g}\steq{def} \int_{0}^{+\infty}g\left(s,\left(\frac{F_{j,N}(Ns)}{N}\right),Z_N(Ns)\right)\diff s,
\end{equation}
for  a continuous function $g$ with compact support on $H$, 
\end{definition}
Note that the ``slow'' process $(\overline{F}_N(t))$ is included in the definition of the occupation measure $\Lambda_N$. The reason is that the timescale is too fast, of the order of $N^2$ in fact,  to get directly convenient tightness properties  for the sequence of processes $(\overline{F}_{j,N}(t))$.

We can have a glimpse of this problem as follows. If $(\overline{M}_{j,N}(t))$ is the martingale of Relation~\eqref{SDE1}, it does not clearly converges in distribution to $0$ as $N$ gets large as it could be expected if a ``standard'' averaging principle were true.  Indeed Relation~\eqref{CrocSDE1} gives, for $1{\le}j{\le}J$, 
\[
\left(\croc{\overline{M}_{j,N}}(t)\right)=
\left(\eta_j\int_{0}^t  \left(\frac{C_{j,N}}{N}{-}\overline{F}_{j,N}(s)\right)\diff s{+}\lambda_j \int_0^t \overline{F}_{j,N}(s)Z_N(Ns)\diff s\right),
\]
which does not seem to vanish. 

We state the main result of this paper. 
\begin{theorem}[Averaging Principle]\label{TheoLLN}
Under the scaling assumption~\eqref{ScaleN} and if $(F_N(0)/N)$ converges to $\overline{f}_0{\ne}0$, then  the sequence of processes $(\|F_N(Nt)\|/N)$ converges in distribution to $(H(t))$, defined by, for $t{\ge}0$, $H(t){\in}(0,1)$ is the unique solution of the relation
\begin{equation}\label{NormLim}
  \int_{\|\overline{f}_0\|}^{H(t)}\frac{1}{\delta \phi(u){-}\beta}\diff u=t,
\end{equation}
where $\phi$ is defined by Relation~\eqref{phi}.

Furthermore the sequence $(\Lambda_N)$ is converging in distribution to the measure $\Lambda_\infty$ on $H{=}\R_+{\times}[0,1]^J{\times}\N$, such that
\begin{equation}\label{OccLim}
\croc{\Lambda_\infty,g}\steq{def} \int_{\R_+{\times}\N} g\left(s,\left(f_j(s)\right),x\right)P_{\phi(H(s))}(\diff x)\diff s,
\end{equation}
for  a non-negative Borelian function $g$ on $H$,  where, for $1{\le}j{\le}J$,
\begin{equation}\label{LimFi}
f_j(t)=\frac{\rho_j}{\rho_j{+}\phi(H(t))}c_j,
\end{equation}
and, for $y{>}0$, $P_y$ is a Poisson distribution with parameter $y$. 
\end{theorem}
Note that we have a convergence in distribution $(\|F_N(Nt)\|/N)$, but not of the processes $(F_{j,N}(Nt)/N)$, $j{=}1$,\ldots, $J$.. The convergence in distribution for this $J$-dimensional process is weaker, it is expressed through the sequence of occupation measures $(\Lambda_N)$. See~\citet{Dawson} for general definitions and results for the convergence in distribution of random measures.

It is not difficult to see that, under Condition~\eqref{InitOp} for the initial conditions, one cannot have a convergence in distribution of $(F_{j,N}(Nt)/N)$. Otherwise, its limit would be $(f_j(t))$, but this would imply that the asymptotic initial point $(\overline{f}_{0,j})$ would satisfy the relation
\[
\overline{f}_{0,j}=\frac{\rho_j}{\rho_j{+}\phi(\overline{f}_0)}c_j,
\]
which is not the case a priori. Asymptotically the vector $(F_{j,N}(Nt)/N)$ lives in a one dimensional curve of the state space, this is due to the fast processes which lead to a kind of state space collapse. See Propositions~\ref{PropPi} and~\ref{TechProp}. 

\begin{corollary}[Equilibrium]\label{corotheo}
Under the assumptions, and with the notations of,  Theorem~\eqref{TheoLLN}, for $1{\le}j{\le}J$,
  \[
\lim_{t\to+\infty} H(t)=H_\infty=\sum_{j=1}^J\frac{\rho_j}{\rho_j{+}\beta/\delta}c_j,
\]
\end{corollary}
The quantity $H_\infty$ is the asymptotic fraction of free particles.
The proof of this theorem is achieved in several steps. The general picture is that nevertheless a kind of averaging principle holds, the slow process being $(\overline{F}_N(t))$ and the ``fast'' process is $(Z_N(Nt))$. The general method in this domain is described in~\citet{Kurtz}, see also~\citet{Papanicolaou} and~\citet{Freidlin}. It turns out that, due to the very fast timescale mentioned above, the slow process has to be included in the definition of the occupation measure, see Definition~\ref{ScalDef}.  This situation leads to technical difficulties, to identify the invariant measures of fast processes in particular.

\begin{definition}\label{deftau}
For $a{>}0$, the stopping time $\tau_N(a)$ is defined by 
\begin{equation}\label{tau}
\tau_N(a)\steq{def}\inf\left\{t{>}0: \|F_N(Nt)\|{=}\sum_{j=1}^J F_{j,N}(Nt) \le a N\right\}.
\end{equation}
\end{definition}
To prove convenient tightness properties of a scaled version of $(X_N(t))$,  we first derive some technical results. 
In a first step, we fix some $a_0{\in}(0,\|\overline{f}_0\|)$ and  we will work with a ``stopped'' {\em occupation measure}, it is the random measure on $H$ defined by
\begin{equation}\label{OccMeasDef}
  \croc{\Lambda^0_N,g}\steq{def} \int_{0}^{\tau_N(a_0)}g\left(s,\left(\frac{F_{j,N}(Ns)}{N}\right),Z_N(Ns)\right)\diff s,
\end{equation}
for  a continuous function $g$ with compact support on $H$.
The motivation of the stopped occupation is due to a technical argument used for the identification of the invariant distributions of fast processes. See Proposition~\ref{TechProp}. 
\begin{lemma}\label{lem1Op}
If $\|\overline{f}_0\|{>}0$, for $a{\in}(0,\|\overline{f}_0\|)$, we have
  \[
  \lim_{N\to+\infty} \P\left(\tau_N(a)< \ell(a)\right)=0,
  \]
with $\ell(a)\steq{def}{\|\overline{f}_0\|{-}a}/{(2\beta)}$, and the relation
\[
\lim_{N\to+\infty} \left(\frac{Z_N(Nt)}{\sqrt{N}}\right)=(0)
\]
 holds for the convergence in distribution. 
\end{lemma}
\begin{proof}
For $t{>}0$ and $N$ sufficiently large, on the event $\{\|F_N(t)\|{<}a N \}$ there are at least $(\|F_N(0)\|{-}\lceil aN\rceil{-}z_0)$ new agents created up to time $Nt$. Consequently for  $y{>}0$, 
  \[
  \{\tau_N(a)\leq y\}\subset \left\{{\cal P}_z^+((0,\beta){\times}(0,yN)\}{\ge}\|F_N(0)\|{-}\lceil a N\rceil{-}z_0\right\},
  \]
  by Relation~\eqref{SDEd2}. The first assertion follows from the law of large numbers for Poisson processes. 

We now show that there exists a coupling of $(Z_N(t))$ with  $(L_0(t))$, the state process  of an  $M/M/\infty$ queue such that  the relation $Z_N(t){\le}(L_0(N^2 t))$ holds for all $t{<}\tau_N(a)$. See Section~\ref{MMI} of the appendix on the $M/M/\infty$ queue.

In state $z{\in}\N$, the jump rates of the process $(Z_N(t))$ in state $((f_j),z)$ at time $t$ are given by 
  \[
 \begin{cases}
{+}1,& \displaystyle \beta{+}\sum_{j=1}^J \eta_j\left(C_j^N{-}f_j\right),\\
{-}1,& \displaystyle \delta z{+}\sum_{j=1}^J \lambda_j f_jz. 
  \end{cases}
 \]
 Let $(L_0(t))$ the  process of the $M/M/\infty$ queue   with input rate $2\overline{\eta}$ and service rate $\delta{+}a\underline{\lambda}$,  with Definition~\eqref{overeta},   with $L_0(0){=}Z_N(0){=}z_0$. We take   $N$ sufficiently large so that $\beta{\le}\overline{\eta}N$. By comparing the jump rates, one can construct a version of $(Z_N(t),L_0(t))$ such that the relation
\begin{equation}\label{ineq1}
 Z_N(Nt)\le L_0(N^2t)
\end{equation}
 holds for all $t{<} \tau_N(a)$.  For $\eps{>}0$, let $T_N(\eps){=}\inf\left\{t: L_0(t){\ge}\eps \sqrt{N}\right\}$. 
Proposition~\ref{HitMMI} of the appendix shows that the sequence of random variables
\[
\left(\left(\frac{2\overline{\eta}}{a\underline{\lambda}}\right)^{\lceil \eps \sqrt{N}\rceil}\frac{T_N(\eps)}{(\lceil \eps \sqrt{N}\rceil{-}1)!}\right)
\]
converges in distribution to an exponential random variable.  In particular, for any $t{>}0$,
\[
\lim_{N\to+\infty} \P(T_N(\eps){\le}N^2 t)=0.
\]
The proof of the lemma  follows from  the relation 
\[
\P\left(\sup_{s{\le}t}\frac{Z_N(N(s{\wedge}\tau_N(a_0)))}{\sqrt{N}}>\eps\right)\le  \P\left(\sup_{s{\le}t}\frac{L_0(N^2s)}{\sqrt{N}}>\eps\right){=}\P\left(T_N{\le}N^2 t\right).
\]
\end{proof}

\begin{proposition}\label{PropLa}
The sequence of random measures $(\Lambda^0_N)$  defined by Relation~\eqref{OccMeasDef} is tight and any limiting point $\Lambda^0_\infty$ can be expressed as
\begin{equation}\label{SAPpi}
\croc{\Lambda^0_\infty,f}=\int_{\R_+{\times}[0,1]^J{\times}\N} f\left(s,x,p\right)\pi_s(\diff x,\diff p)\diff s,
\end{equation}
for any function $f{\in}{\cal C}_c(\R_+{\times}[0,1]^J{\times}\N)$, where  $(\pi_s)$ is an optional process with values in the space of  probability distributions on $[0,1]^J{\times}\N$.

If $(\Lambda^0_{N_k})$ is a sub-sequence of $(\Lambda^0_N)$ converging to $\Lambda^0_\infty$, then with the convention of  Relation~\eqref{ConvInt}, for the convergence in distribution of processes, 
\begin{equation}\label{SAPpicv}
  \lim_{k\to+\infty} \left(\int g(x,z)z\Lambda^0_{N_k}([0,t],\diff x,\diff z)\right)=
  \left(\int g(x,z)z \Lambda^0_{\infty}([0,t],\diff x,\diff z)\right),
\end{equation}
for  any  bounded continuous function  $g$ on $[0,1]^J{\times}\N)$ and the limit is integrable for all $t{\ge}0$. 
\end{proposition}
\begin{proof}
For $K{>}0$, with Relation~\eqref{ineq1} in the proof of  Lemma~\ref{lem1Op}, and with the same notations,  the relation
\[
\int_0^t\ind{Z_N(Ns)\ge K}\diff s \le \int_0^t\P(L_0(N^2s)\ge K)\diff s
\]
 holds  on the event $\{\tau_N(a_0){\ge}t\}$.  Consequently, for  $t{<}\ell(a_0)$,
\begin{multline*}
  \E\left( \Lambda^0_N([0,t]{\times}[0,1]^J{\times}[K,{+}\infty])\right)=
  \E\left(\int_0^{t{\wedge}\tau_N(a_0)}\ind{Z_N(Ns)\ge K}\diff s\right)  \\
  \le\E\left(\ind{\tau_N(a_0)>t}\int_0^{t}\ind{L_0(N^2s)\ge K}\diff s\right)
  \le \frac{1}{N^2}\int_0^{N^2t}\P(L_0(s)\ge K)\diff s.
\end{multline*}
Since the Markov process $(L_0(t))$ converges in distribution to a Poisson distribution with parameter~$2\overline{\eta}/(\underline{\lambda} a_0)$, see Section~\ref{MMI} of the appendix.  With Lemma~\ref{lem1Op}, we obtain the relation
\[
\limsup_{N\to+\infty}   \E\left( \Lambda^0_N([0,t]{\times}[0,1]^J{\times}[K,{+}\infty])\right)\le
\P({\cal N}_1(0,2\overline{\eta}/(\underline{\lambda} a_0)){\ge}K) \,t,
\]
where ${\cal N}_1$ is a   Poisson process on $\R_+$ with rate $1$. In particular, one can choose $K$ sufficiently large such that
\[
\sup_N \E\left( \Lambda^0_N([0,t]{\times}[0,1]^J{\times}[K,{+}\infty])\right)
\]
is arbitrarily small. Lemma~1.3 of~\citet{Kurtz} shows that the sequence $(\Lambda^0_N)$ is  tight, and Lemma~1.4 of the same reference gives the representation~\eqref{SAPpi}.

For the second part of the proposition, Relation~\eqref{ineq1} in the proof of Lemma~\ref{lem1Op} and the Cauchy-Schwartz' Inequality give, for $s{\le}t$, 
\begin{multline*}
\E\left(\left(\int g(x,z)z\Lambda^0_{N}([s,t],\diff x,\diff z)\right)^2\right)
=\E\left(\left(\int_{s{\wedge}\tau_N(a_0)}^{t{\wedge}\tau_N(a_0)}\hspace{-3mm}g\left(\overline{X}_N(s)\right)Z_N(Ns)\diff s\right)^2\right)\\
\le \|g\|_\infty \E\left(\left(\int_{s{\wedge}\tau_N(a_0)}^{t{\wedge}\tau_N(a_0)}L_0(N^2s)\diff s\right)^2\right)
\le \|g\|_\infty \E\left(\left(\int_{s}^{t}L_0(N^2s)\diff s\right)^2\right)\\
\le (t{-}s) \|g\|_\infty\int_{s}^{t}\E\left(L_0(N^2s)^2\right)\diff s\le (t{-}s)^2 \|g\|_\infty \sup_{u{\ge}0}\E\left(L_0(u)^2\right).
\end{multline*}
 Kolmogorov-\v{C}entsov's criterion,  implies that the sequence of processes
\[
\left(\int g(x,z)z\,\Lambda^0_{N}([0,t],\diff x,\diff z)\right)
\]
is tight for the convergence in distribution and any of its limiting points is a continuous process.
See Theorem~2.8 and Problem~4.11, page~64 of~\citet{Karatzas} for example.

For $t{>}0$ and $C{>}0$, for the convergence in distribution we have
\[
\lim_{k\to+\infty}\int g(x,z)\left(z{\wedge}C\right)\,\Lambda^0_{N_k}([0,t],\diff x,\diff z)=\int g(x,z)\left(z{\wedge}C\right)\,\Lambda^0_{\infty}([0,t],\diff x,\diff z).
\]
Using again  Relation~\eqref{ineq1}, with the same argument as before, 
\begin{multline*}
\E\left(\int g(x,z)\left(z{\wedge}C\right)\,\Lambda^0_{N_k}([0,t],\diff x,\diff z)\right)\\
\le \|g\|_\infty\int_0^t \E\left(L_0(N^2u){\wedge}C\right)\diff u
\le t\|g\|_{\infty}\sup_{u{\ge}0} \E\left(L_0(u)^2\right)<{+}\infty,
\end{multline*}
by letting first  $k$  and then $C$ go to infinity, we obtain the relation
\[
\E\left(\int g(x,z) z\,\Lambda^0_{\infty}([0,t],\diff x,\diff z) \right) <{+}\infty,
\]
for all $t{\ge}0$. Similarly, we have
\begin{multline*}
\E\left(\int g(x,z)z\ind{z{\ge}C}\,\Lambda^0_{N_k}([0,t],\diff x,\diff z)\right)\\
\le \|g\|_{\infty}\int_0^t \E\left(L_0(N_k^2u)\ind{L_0(N_k^2u){\ge}C}\right)\diff u
\le \frac{t}{C}\|g\|_{\infty}\sup_{u{\ge}0} \E\left(L_0(u)^2\right),
\end{multline*}
and  therefore the convergence in distribution
\[
\lim_{k\to+\infty}\int g(x,z) z\,\Lambda^0_{N_k}([0,t],\diff x,\diff z)=\int g(x,z)z\,\Lambda^0_{\infty}([0,t],\diff x,\diff z),
\]
for $t{>}0$. For $p{\ge}1$ and $0{\le}t_1{\le}\cdots{\le}t_p$, this convergence also clearly holds for  finite marginals at $(t_i)$. The proposition is proved. 
\end{proof}
If $f$ is a non-negative Borelian function on $\R_+^J{\times}\N$, the SDEs~\eqref{SDEd1} and~\eqref{SDEd2} give directly the relations
\begin{align}
f\left(\overline{F}_N(t),Z_N(Nt)\right)&=f\left(\overline{F}_{N}(0),Z_N(0)\right){+}M_{f,N}(t)\label{SDEf}\\
+\sum_{j=1}^J \lambda_j \int_0^{t}& N\Delta_{-e_j/N,-1}(f)\left(\overline{F}_{N}(s),Z_N(Ns)\right) F_{j,N}(Ns)Z_N(Ns)\diff s\notag\\
  {+}\sum_{j=1}^J \eta_j \int_0^{t} &N\Delta_{e_j/N,1}(f)\left(\overline{F}_N(s),Z_N(Ns)\right)(C_j^N{-}F_{j,N}(Ns))\diff s\notag\\
+\beta N\int_0^{t} &\Delta_{0,1}(f)\left(\overline{F}_N(s),Z_N(Ns)\right)\diff s\notag\\
+\delta N \int_0^{t} &\Delta_{0,-1}(f)\left(\overline{F}_N(s),Z_N(Ns)\right)Z_N(Ns)\diff s,\notag
\end{align}
with the notation, for $x$, $u{\in}\R_+^J$ and $y$, $b{\in}\N$,
\[
\Delta_{u,v}(f)(x,y)\steq{def} f(x{+}u,y{+}v){-}f(x,y),
\]
and,  for $1{\le}j{\le}J$,  $e_j$ is the $j$th unit vector of $\R^J$.

The process  $(M_{f,N}(t))$ is a square integrable martingale whose previsible increasing process is given by
\begin{align}
  &\croc{M_{f,N}}(t)  =\beta N\int_0^{t}  \Delta_{0,1}(f)\left(\overline{F}_N(s),Z_N(Ns)\right)^2\diff s\label{CrocMf} \\
    &\qquad{+}\delta N\int_0^{t}\Delta_{0,-1}(f)\left(\overline{F}_N(s),Z_N(Ns)\right)^2Z_N(Ns)\diff s\notag\\ 
&\qquad{+}\sum_{j=1}^J \lambda_j  \int_0^{t} N\Delta_{-e_j/N,-1}(f)\left(\overline{F}_N(s),Z_N(Ns)\right)^2F_{j,N}(Ns)Z_N(Ns)\diff s\notag\\
    &\qquad{+}\sum_{j=1}^J \eta_j \int_0^{t} N\Delta_{e_j/N,1}(f)\left(\overline{F}_N(s),Z_N(Ns)\right)^2 (C_j^N{-}F_{j,N}(Ns))\diff s.\notag 
\end{align}

If we divide by $N$ Relation~\eqref{SDEf} taken at $Nt$, we get 
\begin{align}
\frac{1}{N}&\left(\rule{0mm}{4mm}f\left(\overline{F}_N(t),Z_N(Nt)\right)\right.{-}\left.\rule{0mm}{4mm}f\left(\overline{F}_N(0),Z_N(0)\right)\right){=}\frac{1}{N}M_{f,N}(t)\label{SDEf2}\\
&{+}\sum_{j=1}^J \lambda_j \int_0^{t} N\Delta_{-e_j/N,-1}(f)\left(\overline{F}_N(s),Z_N(Ns)\right) \overline{F}_{j,N}(s)Z_N(Ns)\diff s\notag\\
&  {+}\sum_{j=1}^J \eta_j \int_0^{t} N\Delta_{e_j/N,1}(f)\left(\overline{F}_N(s),Z_N(Ns)\right)\left(\frac{C_j^N}{N}{-}\overline{F}_{j,N}(s)\right)\diff s\notag\\
&{+}\beta \int_0^{t} \Delta_{0,1}(f)\left(\overline{F}_N(s),Z_N(Ns)\right)\diff s\notag\\
     &{+}\delta\int_0^{t} \Delta_{0,-1}(f)\left(\overline{F}_N(s),Z_N(s)\right)Z_N(Ns)\diff s. \notag
\end{align}
\begin{lemma}\label{LemMart}
  If $f$ is a  bounded ${\cal C}_1$ function with compact support on $\R_+^J{\times}\N$, then the sequence of  martingales $(M_{f,N}(t)/N, t{<}\ell(a_0))$ of Relation~\eqref{SDEf2} converges in distribution to $0$. 
\end{lemma}
\begin{proof}
We take care of the third term $(A_N(t))$ of $(\croc{M_{f,N}/N}(Nt))$ in Relation~\eqref{CrocMf},  the arguments are similar for the others, even easier. For $1{\le}j{\le}J$, denote by $(A_{j,N}(t))$ the $j$th term of $(A_N(t))$, then
\[
A_{j,N}(t)\steq{def}\frac{\lambda_j}{N} \int_0^{t} N\Delta_{-e_j/N,1}(f)\left(\overline{F}_N(s),Z_N(Ns)\right)^2\overline{F}_{j,N}(s)Z_N(Ns)\diff s,
\]
then, again with Relation~\eqref{ineq1} of the proof of  Lemma~\ref{lem1Op}, since $(\overline{F}_{j,N}(t))$ is bounded by $1$, 
\[
\E(A_{j,N}(t))\leq 
\frac{\lambda_j}{N} \left\|\frac{\partial f}{\partial x_j}\right\|_\infty^2\left(\E\left(\int_0^{t}L_0(N^2s)\diff s\right){+}t\P(\tau_N(a_0)\le\ell(a_0))\right),
\]
since $(\E(L_0(t))$ is converging as $t$ goes to infinity, we have
\[
\lim_{N\to+\infty} \E(A_{j,N}(t))=0.
\]
Doob's Inequality shows the convergence of $(M_{f,N}(t))/N,t{<}\ell(a_0))$ to $0$.
\end{proof}

\begin{proposition}\label{pirep}
If $\Lambda^0_\infty$ is a limiting point of $(\Lambda^0_N)$ with the representation~\eqref{SAPpi} of Proposition~\ref{PropLa} then, for any continuous function $g$ on  $\R_+^J{\times}\N$, the relation
\begin{multline}\label{SAPpi2}
\left(\int_0^t \int_{[0,1]^J{\times}\N} g(x,p)\pi_s(\diff x,\diff p)\diff s, t{<}\ell(a_0)\right)\\ =\left(\int_0^t \int_{[0,1]^J} \E\left(g\left(x,{\cal N}_1\left(\left[0,\frac{\croc{\eta,c{-}x}}{\croc{\lambda,x}}\right]\right)\right)\right)\pi_s^{1}(\diff x)\diff s,t{<}\ell(a_0)\right),
\end{multline}
holds almost surely,  where    $\pi^{1}_t$ is the marginal of $\pi_t$ on $\R_+^J$, $\lambda$, $\eta$ and $c{\in}\R_+^J$ are given by Definition~\ref{rho} and $\ell(a_0)$ in Lemma~\ref{lem1Op}, and  ${\cal N}_1$ is a   Poisson process on $\R_+$ with rate $1$.

Furthermore, almost surely, 
\begin{equation}\label{Lowpi}
\int_0^{\ell(a_0)} \pi_s^1\left(x{\in}[0,1]^J: \sum_{j=1}^J x_j< a_0\right)\diff s=0.
\end{equation}
\end{proposition}
Relation~\eqref{SAPpi2} states that, for almost all  $t{<}\ell(a_0)$,  $\pi_t$ conditioned on the first coordinate $x$ is a Poisson distribution with parameter $\croc{\eta,c{-}x}/\croc{\lambda,x} $. Note that Relation~\eqref{Lowpi} shows that the denominator $\croc{\lambda,x}$ is lower bounded by $\underline{\lambda} a_0$ almost surely for $\pi_t$, for $t{<}\ell(a_0)$. 
\begin{proof}
Let $(\Lambda^0_{N_k})$ be a subsequence of $(\Lambda^0_N)$ converging to some $\Lambda^0_\infty$ of the form~\eqref{SAPpi}. 
By letting $k$ go to infinity in Relation~\eqref{SDEf2}, with Lemma~\ref{lem1Op}, Relation~\eqref{SAPpicv} of Proposition~\ref{PropLa}, and Lemma~\ref{LemMart},  for any continuous function $g$ with compact support on $[0,1]^J{\times}\N$, 
\begin{multline*}
\int_{0}^t \int_{[0,1]^J{\times}\N} (g(x,p{-}1){-}g(x,p))\left(\sum_{j=1}^J \lambda_j x_j \right)p \pi_s(\diff x,\diff p)\diff s
\\+ \int_0^t \int_{[0,1]^J{\times}\N} (g(x,p{+}1){-}f(x,p))\left(\sum_{j=1}^J\eta_j (c_j{-}x_j)\right)\pi_s(\diff x,\diff p)\diff s=0,
\end{multline*}
holds almost surely  for all  $t{<}\ell(a_0)$. Hence   we have 
\begin{multline*}
\int_{[0,1]^J{\times}\N}  \lambda (g(x,p{-}1){-}g(x,p))\croc{\lambda,x}p\pi_t(\diff x,\diff p)
\\+ \int_{[0,1]^J{\times}\N} \eta(g(x,p{+}1){-}f(x,p))\croc{\eta,c{-}x}\pi_t(\diff x,\diff p)=0,
\end{multline*}
almost everywhere on $t{\in}\R_+$, or if $g(x,p){=}g_1(x)g_2(p)$,
\[
\int_{\R_+{\times}\N}  g_1(x)\Omega[x](g_2)(p)\pi_t(\diff x,\diff p)=0,
\]
where, for $h:{\N}{\to}\R_+$,
\[
\Omega[x](h)(p)\steq{def}\croc{\eta,c{-}x}(h(p{+}1){-}h(p)){+}\croc{\lambda,x} p (h(p{-}1){-}h(p)).
\]
Let $\pi_t(\cdot\mid x)$ be the conditional probability measure $\pi_t(\diff x,\diff p)$ given $x{\in}\R_+$, then we have, $\pi_t(\diff x,\N)$ almost everywhere
\[
\int\Omega[x](g_2)(p)\pi_t(\diff p\mid x)=0,
\]
for all functions $g_2$ with finite support. Since, for $x{>}0$, $\Omega[x]$ is the $Q$-matrix of an $M/M/\infty$ queue, the last relation  gives that  $\pi_t(\diff p{\mid}x)$ is its invariant distribution, i.e.  it is a Poisson distribution with parameter $\croc{\eta,c{-}x}/\croc{\lambda,x}$.

Relation~\eqref{Lowpi} is simple a consequence of Lemma~\ref{lem1Op}. 

The proposition is proved.

\end{proof}
From now on, we fix $(N_k)$  an increasing sequence such the sequence $(\Lambda^0_{N_k})$ is converging to $\Lambda^0_\infty$ with a representation given by Relations~\eqref{SAPpi} and~\eqref{SAPpi2}.   The following corollary is a direct consequence of Propositions~\ref{PropLa} and~\ref{pirep}. 
\begin{corollary}\label{CVOcc1}
If $g$ is a continuous bounded function on $[0,1]^J$, then, for the convergence in distribution, 
\begin{multline}\label{CVpi2}
\lim_{k\to+\infty}  \left(\int_0^t g\left(\overline{F}_{N_k}(s)\right) Z_{N_k}(N_ks)\diff s, t{<}\ell(a_0)\right)\\
=\left(\int_0^t\int_{[0,1]^J}g(x)\frac{\croc{\eta,c{-}x}}{\croc{\lambda,x}}\pi_s^{1}(\diff x)\diff s, t{<}\ell(a_0)\right).
\end{multline}
\end{corollary}

The following proposition is a convergence result for the scaled process $(\overline{F}_N(t))$. It is weaker that the convergence stated in Theorem~\ref{TheoLLN} clearly, but this is a crucial ingredient to establish the theorem in fact. 
\begin{proposition}\label{NormF}
The sequence of processes $(\|\overline{F}_{N_k}(t)\|, t{<}\ell(a_0))$ is converging in distribution to
\[
(H(t), t{<}\ell(a_0))\steq{def} \left(\|\overline{f}_0\|{+}\int_0^t\int_{[0,1]^J} \left(\delta\frac{\croc{\eta,c{-}x}}{\croc{\lambda,x}}{-}\beta \right)\pi_s^{1}(\diff x)\diff s, t{<}\ell(a_0)\right).
\]
\end{proposition}
\begin{proof}
  We define, for $t{\ge}0$, 
\begin{equation}\label{eqd1}
\widetilde{Z}_N(t)=Z_N(t){+}\sum_{j=1}^J (N_j{-}F_{j,N}(t))= N{-}\|F_N(t)\|{+}Z_N(t),
\end{equation}
$\widetilde{Z}_N(t)$ is the total number of agents (free or paired)  of the system at time $t$.  Using the SDEs~\eqref{SDE1} and~\eqref{SDE2},
 we have
\begin{equation}\label{eqZb}
\frac{\widetilde{Z}_N(Nt)}{N}=M_{Z,N}(t){+}\frac{\widetilde{Z}_N(0)}{N}{+}\beta t{-}\delta\int_0^tZ_N(Ns)\diff s,
\end{equation}
where $(M_{Z,N}(t))$ is a local martingale whose previsible increasing process is given by
\begin{equation}\label{eqZMcroc}
\left(\croc{M_{Z,N}}(t)\right)=\left(\frac{1}{N}\left(\beta t{+}\delta\int_0^tZ_N(Ns)\diff s\right)\right).
\end{equation}
Using again Doob's Inequality, Lemma~\ref{lem1Op} and Relation~\eqref{eqd1}, the proposition is proved.
\end{proof}
The next proposition gives a characterization of the process $(\pi^1_s)$ which will be elucidated in Proposition~\ref{TechProp}.
\begin{proposition}\label{PropPi}
If $g$ is a ${\cal C}_1$-function on $[0,1]^J$, then, almost surely, 
\begin{equation}\label{EqDefPi}
\left(\int_0^{t} \int_{[0,1]^J}\hspace{-1mm}\sum_{j=1}^J\hspace{-1mm} \frac{\partial g}{\partial x_j}(x)\hspace{-1mm}  \left(\lambda_jx_j \frac{\croc{\eta,c{-}x}}{\croc{\lambda,x}}{-}\eta_j(c_j{-}x_j)\right)\pi_s^{1}(\diff x)\diff s, t{<}\ell(a_0)\right){=}(0).
\end{equation}
\end{proposition}
\begin{proof}
For $t{>}0$, let $g$ be a ${\cal C}_2$-function on $[0,1]^J$, Relation~\eqref{SDEf2} gives, 
\begin{multline}\label{ChAEq1}
\frac{1}{N}g\left(\overline{F}_N(t)\right)=\frac{1}{N} g\left(\overline{F}_{N}(0)\right){+}M_{g,N}(t)\\
+\sum_{j=1}^J \lambda_j \int_0^{t} N\nabla_{-e_j/N}(g)\left(\overline{F}_{N}(s)\right) \overline{F}_{j,N}(s)Z_N(Ns)\diff s\\
  {+}\sum_{j=1}^J \eta_j \int_0^{t} N\nabla_{e_j/N}(g)\left(\overline{F}_N(s)\right)\left(\frac{C_j^N}{N}{-}\overline{F}_{j,N}(s)\right)\diff s,
\end{multline}
and, since
\begin{multline*}
\left(\croc{M_{g,N}}(t)\right)
=\left(\frac{1}{N^2}\sum_{j=1}^J \lambda_j \int_0^{t} \left(\rule{0mm}{4mm}N\nabla_{-e_j/N}(g)\left(\overline{F}_{N}(s)\right)\right)^2\overline{F}_{j,N}(s)Z_N(Ns)\diff s
\right. \\\left.  {+}\frac{1}{N^2}\sum_{j=1}^J \eta_j\int_0^{t}  \left(\rule{0mm}{4mm}N\nabla_{e_j/N}(g)\left(\overline{F}_{N}(s)\right)\right)^2\left(\frac{C_j^N}{N}{-}\overline{F}_{j,N}(s)\right)\diff s\right),
\end{multline*}
Relation~\eqref{SAPpicv} shows that, for $t{<}\ell(a_0)$, $\E(\croc{M_{g,{N_k}}}(t))$ is converging to  $0$,  the martingale $(M_{g,{N_k}}(t))$ converges therefore in distribution to $0$ as $k$ gets large. By using again Relation~\ref{SAPpicv} and the differentiability properties of $g$, it is easy to conclude the proof of the proposition. 
\end{proof}
For $1{\le}j{\le}J$, by taking a function $g(x){=}h(x_j)$, $x{=}(x_i){\in}[0,1]^J$, where $h$ is ${\cal C}_2$, we obtain, that the relation
\[
\left(\int_0^{t} \int_{[0,1]^J} h'(x_j)  \left(\lambda_jx_j \frac{\croc{\eta,c{-}x}}{\croc{\lambda,x}}{-}\eta_j(c_j{-}x_j)\right)\pi_s^{1}(\diff x)\diff s, t{<}\ell(a_0)\right)=(0)
\]
holds almost surely. Hence, almost surely, for any $f$ in a dense subset of Borelian functions on $[0,1]$, the identity
\[
\int_{[0,1]^J} f(x_j)  \left(\lambda_jx_j \frac{\croc{\eta,c{-}x}}{\croc{\lambda,x}}{-}\eta_j(c_j{-}x_j)\right)\pi_t^{1}(\diff x)
\]
holds almost everywhere for $t{\in}\R_+$, with respect to Lebesgue's measure.

If $U(t){=}(U_j(t))$ is a random variable with distribution $\pi_t^1$, the last relation can be translated in terms of conditional expectation, almost surely 
\[
\lambda_jU_j(t)\E\left.\left(\frac{\croc{\eta,c{-}U(t)}}{\croc{\lambda,U(t)}}\right| U_j(t) \right)=\eta_j(c_j{-}U_j(t)),
\]
almost everywhere for $t{\ge}0$. The following proposition is the key step to identify completely the limit points of $((\|\overline{F}_N(t)\|),\Lambda^0_N)$.
\begin{proposition}\label{TechProp}
Let $U{=}(U_j)$ be a random variable on  $\prod_1^J[0,c_j]$, such that ,  almost surely, $\|U\|{\ge}\eta$, for some $\eta{>}0$, and,  for $1{\le}j{\le}J$, 
\begin{equation}\label{eqCond}
\lambda_jU_j\E\left.\left(\frac{\croc{\eta,c{-}U}}{\croc{\lambda,U}}\right| U_j \right)=\eta_j(c_j{-}U_j),
\end{equation}
then, almost surely,
\[
U_j=\frac{\rho_jc_j}{\phi(\|U\|){+}\rho_j},
\]
where for $y{\in}(0,1)$, $\phi(y)$ is the unique solution of the equation
\[
\sum_{j=1}^J \frac{\rho_j}{\rho_j{+}\phi(y)}c_j=y,
\]
and $(\rho_j)$ is given by Definition~\ref{rho} and $\phi$ in Theorem~\ref{TheoLLN}. 
\end{proposition}
\begin{proof}
  Define, for $1{\le}j{\le}J$,
  \[
  A_j\steq{def} \frac{\eta_j(c_j{-}U_j)}{ \lambda_jU_j}, \quad
  B_j\steq{def} \lambda_jU_j,
  \]
  and ${\cal F}_j$ is the $\sigma$-field associated to $U_j$. Note that, because of the assumptions the variables $A_j$ and $B_j$ are bounded by Relation~\eqref{eqCond}, this identity can be re-written as
  \[
  \E\left.\left(\frac{\sum_k A_kB_k}{\sum_k B_k}\right|{\cal F}_j\right)=A_j,
  \]
  or, since $A_j$ and $B_j$ are ${\cal F}_j$-measurable
  \[
\E\left.\left(\frac{\sum_k (A_k{-}A_j)A_jB_jB_k}{\sum_k B_k}\right|{\cal F}_j\right)=0,
\]
this identity gives therefore the relation
\[
\E\left(\frac{\sum_{j,k} (A_j{-}A_k)^2B_jB_k}{\sum_k B_k}\right)=
2\sum_{j=1}^J\E\left(\frac{\sum_k (A_j{-}A_k)A_jB_jB_k}{\sum_k B_k}\right)=0,
\]
since the variables $B_j$, $j{=}1$,\ldots,$J$ are almost surely positive, this implies that, almost surely,
$A_j{=}A_1$ for $j{=}1$,\ldots,$J$, and therefore
\[
U_j=\frac{\rho_jc_j}{A_1{+}\rho_j},
\]
the proposition is proved. 
\end{proof}

\begin{proof}[Proof of Theorem~\ref{ScalDef}]
We have only to gather all the results obtained up to now. Proposition~\ref{NormF} shows that, on the time interval $[0,\ell(a_0))$, the sequence of processes $(\|F_{N_k}(t)\|)$ converges in distribution to $(H(t))$ given by 
\begin{multline}\label{ChODE}
  (H(t))=\left(\|\overline{f}_0\|{+}\int_0^t\int_{[0,1]^J} \left(\delta\frac{\croc{\eta,c{-}x}}{\croc{\lambda,x}}{-}\beta \right)\pi_s^{1}(\diff x)\diff s\right)
\\  =  \left(\|\overline{f}_0\|{+}\int_0^t(\delta\phi(H(s)){-}\beta)\diff s\right),
\end{multline}
  by Proposition~\ref{TechProp}. This determines completely $(H(t))$ as the solution of Relation~\eqref{NormLim}, and therefore the convergence of the sequence $(\|F_{N}(t)\|)$. Relation~\eqref{EqDefPi} of Proposition~\ref{PropPi} and  Proposition~\ref{TechProp} gives the desired expression. Relation~\eqref{OccLim} for the limit of the sequence  $(\Lambda_N^0)$ of occupation measures.

  It is easily seen that the solution of   Relation~\eqref{NormLim} satisfies Corollary~\ref{corotheo}, in particular, if $H(0){>}0$, there exists $w{>}0$ such that $H(t){\ge} w$ for all $t{\ge}0$. The convergence in distribution obtained can be therefore extended to any finite interval of $\R_+$. The theorem is proved. 
\end{proof}

\printbibliography

\appendix
\addcontentsline{toc}{section}{Biological Background}
\addtocontents{toc}{\protect\setcounter{tocdepth}{0}}

\section{Biological Background}\label{BioSec}
In bacterial cells,  protein  production uses an important number of cell resources:  macro-molecules such as {\em polymerases} and {\em ribosomes},  biological bricks of proteins, i.e.,  amino acids, and  the energy necessary to build proteins, such as GTP. The two main steps associated with  protein and RNA production are
\begin{itemize}
\item {\em Transcription}.  
When an RNA  polymerase is bound to an active gene, it starts to make a copy of this gene. The   product which is a sequence of nucleotides is an {\em  RNA}. If the gene is associated to a protein, it is a {\em messenger} RNA,  an {\em mRNA}.   When the full sequence of nucleotides of the RNA has been  successively assembled,  the RNA is released in the cytoplasm. 
\item[]
\item {\em Translation}. The step is achieved through another large macro{-}molecule: a {\em ribosome}. When a ribosome  is bound to an mRNA, it builds a chain of amino{-}acids  using the mRNA as a template to produce a linked  chain of amino-acids, a {\em  protein}.
\end{itemize}
There are several types of RNAs, outside mRNAs, tRNAs to carry amino-acids for the translation phase, rRNAs which are (large) building blocks of ribosomes. The average size of an mRNA is of the order of 300 nucleotids (nt), the size of an rRNA is of the order of 5000nt. Another class of RNAs, the small RNAs, or sRNAs, has been discovered in the 1970's, the 6S sRNAs in particular. The size of a sRNA is of the order of 50-100nt and their functional role of regulation has been identified around 2000, quite recently in fact.  See~\citet{Hindley} and~\citet{Beisel}. 

In a biological context, the internal dynamics of cells can be, essentially, expressed in terms of  pairing mechanisms of various couples of macro-molecules.  The quite recent discovery of the existence of small RNAs and of their functional role in the regulation of the protein production has shed a new light on pairing mechanisms as a general tool to control gene expression.  Pairing mechanisms are to regulate the growth of cells in order to adapt to a varying environment. Depending on external  conditions, it may be desirable to reduce or speed-up  the growth of the cell, and so the use of these macro-molecules. A specific regulation mechanism relies on a type of macro-molecules, sRNAs, small RNAs, which we will refer to as {\em agents} in the following. Their functional role has been discovered only recently in fact, at the end of the 20th century. They may bind/pair with one of these  macro-molecules and the pairing may have several effect, depending on the type of sRNA:  it can sequester a  macro-molecule/particle and, therefore, reduce significantly its interactions with other components of the cell. Or it may  speed-up the activity of the macro-molecule by increasing its interaction rate (affinity) with other components of the cell.

Due to thermal noise inside the cell, a particle/agent pair will split after a certain time. These separation mechanisms are as important  as binding events are. The quantity of agents presents within the cell is also highly variable, and their quantity is actively regulated by the cell through a variety of regulatory mechanisms. In most prokaryotes, when a cell grows, its size  increases. Without any action, the concentrations of agents/particles inside the cell begin to decrease by dilution. In the case of a limited number of agents, the cell has specific means of degrading agents to rapidly reduce their quantity if dilution is not enough. When the cell environment changes, the number of active agents may have to increase to adapt. It can be done  by either generating new agents and particles or, more rapidly, by altering the quantity of active agents, i.e., agents bound to particles, by enhancing couplings. The duration of these transitions is an important characteristic of the control mechanism under consideration. In general it has two terms. The first term is linked to the production of new agents (and is therefore rather slow), while the second term is linked to a (quick) change in the number of active machines via bindings of agents with particles. The rate of adaptation depends on the combination of these two mechanisms, and is therefore a crucial issue in the case of cells. Assuming that the environment is such that regulation is necessary, the fraction of particles that are bound to an agent is a measure of the efficiency of these mechanisms. We refer to Section~\ref{BioSec} of the appendix for references and further details on the biological context.

We describe several examples of regulation via pairing mechanisms. For the sake of simplicity, we  give a rough description, not all macro-molecules, mechanisms involved are not mentioned. The main goal is of showing that pairing mechanisms play a major role concerning the regulation of the cell.  References with much more details are mentioned.
\subsection*{Regulation of Transcription}
The  6S sRNA is in charge of regulating the transcription phase, via a pairing with a $\sigma$-factor, a macromolecule necessary to the initiation of  transcription. See~\citet{Fromion2} for more details. With a slight abuse, by ignoring the  $\sigma$-factor, the mechanism can be represented as a chemical reaction with three chemical species particles, 
\[
\mathrm{RNAP}{+}\mathrm{sRNA} \xrightleftharpoons{\phantom{aa}} \mathrm{RNAP}{-}\mathrm{sRNA},
\]
where RNAP stands for RNA polymerase. In such a context the pairing of RNAP and sRNA is seen as a sequestration of polymerases. See~\citet{Waters2009}, \citet{Wassarman2000} and \citet{Felden}. See~\citet{Nitzan2014}.

\subsection*{Regulation of Translation}
There is a similar regulation mechanism for ribosomes. The pairing of a macromolecule denoted as (p)ppGpp with the ribosome has the main effect of interfering with the initiation phase of translation and leading to abort the operation. See~\citet{Yang}, \citet{Fer} and~\citet{Hauryliuk} for more details. 

\subsection*{Regulation of mRNAs by sRNAs}
The translation can be also controlled via the mRNAs in the following way. The pairing of an sRNA and an mRNA modifies the translation efficiency of the mRNA. It can repress or enhance its activity.  The pairing of an sRNA and mRNA . See~\citet{Waters2009}, \citet{Jagodnik2017} and \citet{Beisel} for prokaryotic cells, and~\citet{Flynt} for eukaryotic cells, with micro RNAs, miRNAs,  acting on mRNAs. The references~\citet{Jaya},~\citet{Baker2012} and \citet{Giudice} study the Fokker-Planck equations for Markovian models of this mechanism.

\section{Technical Results}\label{TechApp}
\subsection{The $M/M/1$ queue}\label{MM1}
We introduce a birth and death process  on $\N$ which is in fact a reflected random walk on $\N$. 
An $M/M/1$ queue with input rate $\gamma$ and service rate $\mu$  on $\N$ is a Markov process $(L_1(t))$ with $Q$-matrix matrix  given by 
  \[
z\mapsto \begin{cases}
z{+}1,&   \gamma,\\
z{-}1,&  \mu\ind{z{\ge}1}.
\end{cases}
\]
If $\gamma{<}\mu$, its invariant distribution is a geometric distribution with parameter $\gamma/\mu$.
Define
\[
T_K=\inf\left\{t: L_\infty(t){\ge} K\right\},
\]
 Proposition~5.11 of  Chapter~5 of ~\citet{Robert} gives the asymptotic behavior of $T_K$ when $K$ is large. 
\begin{proposition}\label{HitMM1}
If $\gamma{<}\mu$ and  $L(0){=}z_0$, then, as $K$ goes to infinity, the random variable 
\[
\left(\frac{\gamma}{\mu}\right)^{K}T_K
\]
converges in distribution to an exponential random variable. 
\end{proposition}

\subsection{The $M/M/\infty$ queue}\label{MMI}
This is another classical birth and death process  on $\N$. An $M/M/\infty$ queue with input rate $\gamma$ and service rate $\mu$  on $\N$ is a Markov process with $Q$-matrix given by
  \[
z\mapsto \begin{cases}
z{+}1,&   \gamma,\\
z{-}1,&  \mu z.
\end{cases}
\]
Its invariant distribution is a Poisson distribution with parameter $\gamma/\mu$. 

It can be seen in fact as a discrete Ornstein-Uhlenbeck process. See Chapter~6 of~\citet{Robert} for example.
For $K{\in}\N$, the hitting time of level $K$ by  the  process $(L(t))$  of an $M/M/\infty$  input rate $\gamma$ and service rate $\mu$  is defined by 
\[
T_K=\inf\left\{t: L_\infty(t){\ge} K\right\}.
\]
The following result, see Proposition~6.10 of~\cite{Robert} for example, is used to establish several tightness results. 
\begin{proposition}\label{HitMMI}
If $L(0){=}z_0$, as $K$ goes to infinity, the random variable 
\[
\left(\frac{\gamma}{\mu}\right)^{K}\frac{T_K}{(K{-}1)!}
\]
converges in distribution to an exponential random variable. 
\end{proposition}

\end{document}